%% file: final.tex
\documentclass[11pt]{article}
\newcommand{\weight}{degree}
\usepackage{varieties-commands}

\title{Determinants of Subquotients of Galois Representations Associated to Abelian Varieties}
\author{Eric Larson and Dmitry Vaintrob}
\date{}

\begin{document}
\maketitle

\begin{abstract}
Given an abelian variety $A$
of dimension $g$ over a number field~$K$,
and a prime $\ell$,
the $\ell^n$-torsion points of $A$ give rise to a representation
$\rho_{A, \ell^n} \colon \gal(\bar{K} / K) \to \gl_{2g}(\zz/\ell^n\zz)$.
In particular, we get a \emph{mod-$\ell$ representation}
$\rho_{A, \ell} \colon \gal(\bar{K} / K) \to \gl_{2g}(\ff_\ell)$
and an \emph{$\ell$-adic representation}
$\rho_{A, \ell^\infty} \colon \gal(\bar{K} / K) \to \gl_{2g}(\zz_\ell)$.
In this paper, we describe
the possible determinants of subquotients
of these two representations.
These two lists turn out to be remarkably similar.

Applying our results in dimension $g=1$, we recover
a generalized version of a theorem of Momose on
isogeny characters of elliptic curves over number fields,
and obtain, conditionally on the Generalized Riemann
Hypothesis, a generalization of Mazur's bound on
rational isogenies of prime degree to number fields.
\end{abstract}

\section{Introduction}

Let $A$ be a $g$-dimensional abelian variety over a number field $K$.
The \emph{$\ell$-adic Tate module}
\[A[\ell^\infty] := \varprojlim_n A[\ell^n]\]
is the limit of $\ell$-power torsion points over $\bar{K}$.
It is a $\zz_\ell$-lattice with action by the Galois group
$G_K:=\gal(\bar{K}/K)$, and is one of the fundamental
examples of a Galois representation.

We study in this paper one-dimensional Galois characters which can arise
from these representations. Namely, we consider
determinants of subquotients
of the $\ell$-adic Tate
module of $A$ with scalars extended from $\zz_\ell$
to either $\ell$-adic fields
or finite fields $\ff_{\ell^n}$.
Any such determinant character with values in a field $k$
appears after extending scalars all the way to
$\bar{k}$. Hence studying these determinant characters with
scalars extended to $\bar{\qq}_\ell$ gives all
such characters with values in an $\ell$-adic field,
and extending scalars to $\bar{\ff}_\ell$ gives
all such characters with values in
$\ff_{\ell^n}$.

If $V$ is a representation of a group $G$
over a field $k$,
we say that $\psi \colon G \to \bar{k}$ is an
\emph{associated character} of {\weight} $d$ of $V$
if there is a $d$-dimensional subquotient $W$ of $V \otimes_k \bar{k}$
such that $\psi = \det_{\bar{k}} W$.
We call our principal objects of study --- the
associated characters of $A[\ell^\infty] \otimes \qq_\ell$
and $A[\ell^\infty] \otimes \ff_\ell = A[\ell]$ --- the
\emph{$\ell$-adic} and \emph{mod-$\ell$}
associated characters of $A$ respectively.

The study of associated characters of these kinds goes back to Serre's
foundational work on the Open Image Theorem,
which states that
for an elliptic curve $E$ without complex multiplication, the
action of the absolute Galois group of a number field
on the ad\`elic Tate module $H_1(E, \hat{\zz}) = \prod_\ell E[\ell^\infty]$
is open (i.e.\ has finite index) in
$\gl_2(\hat{\zz})$.
This is proved in two principal steps.
First, in \cite{abelianladic}, Serre shows that the $\ell$-adic image
$\rho_{E, \ell^\infty}\colon G_K\to \gl_2(\zz_\ell)$
has finite index for all $\ell$.
Second, in \cite{serre},
Serre shows that for sufficiently large primes,
the mod-$\ell$ image
$\rho_{E, \ell} \colon G_K\to \gl_2(\zz_\ell)$
is surjective.
In each case, the proof consists of reducing the problem to
the study of the $\ell$-adic and mod-$\ell$ associated characters
of $E$ respectively, i.e.\ studying the Galois action on the $1$-dimensional
subquotients.

A major conjecture, which is still open, is whether the index
of the image of $G_K$ in $\gl_2(\hat{\zz})$ in Serre's theorem is
bounded \emph{uniformly} in $E$.
The first step towards this
result, for $K = \qq$, is Mazur's seminal theorem on isogenies \cite{mazur}.
This theorem is equivalent to the statement that, for an elliptic curve
$E$ over $\qq$ and for a prime $\ell > 163$, the $\ell$-torsion
module $E[\ell]$ is irreducible (equivalently, no
isogenies $E\to E'$ defined over $\qq$ have kernel of order $\ell$).
An essential step of Mazur's proof is to analyze
the possible associated characters (up to torsion of small degree)
of subquotients
of $E[\ell]$, and show that for $\ell > 163$
the list of possible associated characters is empty.

Momose in \cite{momose} gives an exhaustion 
(i.e.\ a list containing all possibilities, perhaps with excess)
for the mod-$\ell$ associated characters
of elliptic curves over number fields $K$
attached to subquotients of $E[\ell]$, for $\ell$ sufficiently large
depending on $K$.
When $K$ is quadratic,
either real, or imaginary of class number greater than one
(i.e.\ as long as
$K \neq \qq[\sqrt{D}]$ for
$D \in \{-1, -2, -3, -7, -11, -19, -43, -67, -163\}$),
the list of possible associated characters is empty.
In particular, any elliptic curve $E$ over such
a quadratic field $K$ has irreducible torsion module $E[\ell]$ (equivalently,
admits no $\ell$-isogenies) as long as $\ell>C_K$ for some constant $C_K$
that depends only on $K$.

The main theorem of our paper gives an analogous exhaustion for abelian varieties
of dimension $g$ over $K$.
When applied to elliptic curves, we obtain slightly stronger versions of
the above results of Momose (see Theorem~\ref{introellcurve}
and Corollary~\ref{cor:introisogeny} below).
While it is hopeless to classify all proper subquotients
of $A[\ell]$ for $g > 1$ (for $A$
decomposable, for example, this would involve 
classifying all Galois representations coming from elliptic curves), giving
a complete characterization of their determinants is a more manageable
task --- and this is the question we study in this paper.

Serre's Open Image Theorem is the beginning of a larger story
about the Galois representations $A[\ell]$.
By Faltings' Finiteness Theorem \cite{faltings} (applied
to both $A$ and $A \times A$), we know that
if $\mend_K(A) = \zz$, then the representations
$A[\ell]$ are absolutely irreducible for $\ell$ larger
than some constant depending on $A$. As with Serre's theorem,
conjecturally there should exist a uniform bound.
Our main result implies this conjecture
for a large class of fields, conditionally on the Generalized
Riemann Hypothesis (GRH).
Since our main result holds for
arbitrary $g$, there is hope that it gives a step towards
this conjecture in general.

Our paper consists of two main parts. First we study 
$\ell$-adic associated characters of abelian varieties,
giving an essentially complete classification.
Then, we turn to the
study of mod-$\ell$ characters.
We show that for $\ell$ greater than some constant
depending only on $K$ and $g$, any mod-$\ell$
associated character of a $g$-dimensional
abelian variety defined over $K$ belongs to
a small list of possibilities.
An abelian variety with a mod-$\ell$
associated character does not necessarily have
an $\ell$-adic associated character;
therefore, a priori, there could be many more possibilities
for mod-$\ell$ associated characters than for $\ell$-adic associated characters.
Remarkably, this is not the case:
\emph{Our list of possible mod-$\ell$ characters is
closely related to the list of $\ell$-adic
associated characters that can occur.}

\medskip

Now we will turn to a precise formulation of our results.
The relationship between $\ell$-adic and mod-$\ell$
associated characters discussed above
is particularly sharp in dimension $g=1$ (elliptic curves), 
especially if one assumes GRH. In
this case we have the following theorem.

\begin{ithm}[Theorem~\ref{ellcurve}] \label{introellcurve}
Let $K$ be a number field.
Then, there exists a
finite set $S_K$ of prime numbers depending only on $K$
such that, for a prime $\ell \notin S_K$,
and an elliptic curve $E$ over $K$ for which $E[\ell]\otimes\bar{\ff}_\ell$ 
is reducible with {\weight} $1$ associated character $\psi$,
one of the following holds.
\begin{enumerate}
\item\label{iellcurve1} There exists a CM
elliptic curve $E'$, which is defined over $K$
and whose CM-field is contained in $K$, with
an $\ell$-adic {\weight} $1$ associated character
whose mod-$\ell$ reduction $\psi'$ satisfies:
\begin{equation}\label{iellcurvealign}
\psi^{12} = (\psi')^{12}
\end{equation}
\item
The Generalized Riemann Hypothesis fails for $K[\sqrt{-\ell}]$, and 
\begin{equation} \label{iradical}
\psi^{12} = \cyc_{\ell}^6,
\end{equation}
where $\cyc_\ell$ is the cyclotomic character.
(Moreover, in this case we must have $\ell \equiv 3$ mod $4$ and
the representation $\rho_{E, \ell}$ is already reducible over $\ff_\ell$.)
\end{enumerate}
\end{ithm}

\begin{rem}
The proof of Theorem~\ref{introellcurve} implies that $E'$
depends only on $E$ (and not on $\ell$); moreover,
$\psi' \otimes \psi^{-1}$ is ramified only
at primes of bad additive reduction for $E$.
\end{rem}
The proof of Theorem~\ref{introellcurve} also shows that the set
$S_K$ is effectively computable. In Theorem~\ref{effectivethm}, we
give an explicit bound on the value of $\prod_{\ell \in S_K} \ell$.

In Case~\ref{ellcurve1} of the theorem (which is the only possible
case if we assume GRH),
the twelfth powers of the
other associated characters of $E[\ell]$ and $E'[\ell]$
also coincide,
since the two representations $E[\ell]$ and $E'[\ell]$ have the same
determinants (both equal to $\cyc_\ell$). In particular,
equation \eqref{iellcurvealign}
can be formulated more concisely as
\[\Psi^{12}(E[\ell])=\Psi^{12}(E'[\ell]),\]
where $\Psi^n$ is the Adams operation
on the Grothendieck ring of representations
of $G_K$ over $\ff_\ell$.
Similarly, we can rewrite equation \eqref{iradical} as
$\Psi^{12} (E[\ell]) = \cyc_\ell^6 \oplus \cyc_\ell^6$.

Theorem~\ref{introellcurve} implies the following result.
(This is also a straightforward consequence of the results
of \cite{merel} and \cite{momose}, but does not appear
to be written anywhere in the literature.)
\begin{icor}[Corollary~\ref{isogeny}] \label{cor:introisogeny}
Under GRH,
the degrees of prime degree isogenies of elliptic curves over $K$
are bounded uniformly if and only if $K$ does not contain the Hilbert class field
of an imaginary quadratic field $F$ (i.e.\ if and only if there
are no elliptic curves with CM defined over $K$).
\end{icor}

Theorem \ref{introellcurve} follows from the more general Theorem
\ref{intromainthm}, which gives an analogous
statement for arbitrary abelian varieties. 
Before formulating it, we need to introduce some
technical notions. First, however,
we give a statement in the case
where $K$ has a real embedding, which is considerably
simpler. More generally, it suffices to assume that $K$ does
not contain any \emph{CM-fields}, i.e.\ totally
imaginary quadratic extensions of totally real fields.

\begin{icor}[Corollary~\ref{cor:nocmfield}] \label{cor:intronocmfield}
Let $K$ be a number field that does not contain
any CM-fields (which is in particular true when $K$
has a real embedding), and $g$ and $d$ be positive integers.
There exists a
finite set $S_{K, g}$ of prime numbers depending only on $K$
and $g$, and a constant $0 < c_g < 12^{4g^2}$ depending only on $g$
such that, for a prime $\ell \notin S_{K, g}$,
and a $g$-dimensional abelian variety $A$
with a mod-$\ell$ associated character $\psi$
of {\weight} $d$,
\[\psi^{2w} = \cyc_\ell^{aw},\]
where
$a$ is an integer with $0 \leq a \leq 2d$, and
$w = \frac{\lcm(c_g, N)}{2}$ for some positive $N \leq \binom{2g}{d}$.
\end{icor}

In general, we will try to relate
the associated character $\psi$ to the mod-$\ell$
reduction of an $\ell$-adic associated character
of another abelian variety $B$. This encompases the above
case, because the determinant of the entire
Tate module gives an $\ell$-adic associated
character which is a power of the cyclotomic character,
and there are no $\ell$-adic associated characters
of abelian varieties defined over such fields $K$
besides twists of powers of the cyclotomic character.
In order to do this, we will need to
study the $\ell$-adic associated characters
of abelian varieties first.

Using a theorem of Faltings \cite{faltings}, we will see that any
$\ell$-adic
associated character arises from an abelian variety $B$
which has (generalized) \emph{complex multiplication},
i.e.\ has action by an order in a field $F$, where $F$ is a
CM-field or $F = \qq$.
Note that we do \emph{not} assume that $\deg F = 2 \cdot \dim A$
(as is the case in the classical theory of complex multiplication);
if this is the case, we call it \emph{full} CM.

Suppose that $B$ is an abelian variety over
$K$ with complex multiplication
by a CM field $F$,
i.e.\ with an injection $F\hookrightarrow \mend(B)\otimes\qq$. 
Then
$B[\ell^\infty]\otimes\qq_\ell$ is an $(F\otimes\qq_\ell)$-module, which one
can show is free of dimension $\frac{2g}{[F:\qq]}$. The problem
of finding $\ell$-adic associated characters for the $F$-eigenspaces
reduces to computing
the $(F\otimes\qq_\ell)$-determinant
of this representation --- composing this determinant 
with various embeddings $F\hookrightarrow \bar{\qq}_\ell$
gives these $\ell$-adic associated characters of $B$. 
By the local characterization of such
determinants in Appendix~\ref{detcrys} (written by Brian Conrad),
we see that these characters are determined (up to twists)
by the CM type of $B$, i.e.\ the isomorphism class of the 
$K$-$F$ bimodule
$\Phi = \lie(B)$; they can be described quite
explicitly in terms of ``algebraic'' class 
field theoretic characters of the type studied by Serre in \cite{abelianladic}.

Specifically, given a CM type $(F,\Phi)$ we define a 
field $K_{F, \Phi}$ and a Galois character
$\psi_{F,\Phi} \colon G_{K_{F, \Phi}} \to \bar{\qq}_\ell^\times / \mu_F$
(where $\mu_F$ is the group of roots of unity contained in $F$),
which are uniquely determined by the following properties,
as shown in Theorems~\ref{psifphi} and~\ref{irredthm}.
(Technically, the character $\psi_{F, \Phi}$ depends
on a choice of embedding $\sigma \colon F \hookrightarrow \bar{\qq}_\ell$
as well, but we will usually suppress this and assume we have chosen
an appropriate embedding.)

\begin{enumerate}
\item If an abelian variety $B$ over $K$ has CM of type
$(F, \Phi)$ which is defined over $K$,
then $K_{F, \Phi}$ is contained in the ground field $K$.

\item In the above case, the associated characters
of the $F$-eigenspaces of 
$B[\ell^\infty] \otimes \qq_\ell$ define characters
$\psi_{B, \sigma} \colon G_K \to \bar{\qq}_\ell^\times$
(indexed by embeddings $\sigma \colon F \hookrightarrow \bar{\qq}_\ell$),
and modulo the group $\mu_F$ of roots of
unity in $F$, these characters coincide with the restriction
of $\psi_{F,\Phi}$ to $G_K$, i.e.\ 
\[\psi_{F,\Phi}|_{G_K} \equiv \psi_{B,\sigma} \mod \mu_F.\]

\item \label{shim}
The field $K_{F, \Phi}$ is the \emph{separating field}
(i.e.\ minimal
field of definition of the geometrically irreducible components) of
the Shimura variety for some nonempty
collection of polarized abelian varieties with CM type $(F, \Phi)$.
\end{enumerate}

The Shimura variety in condition~\ref{shim} is
the coarse moduli space for this collection of polarized abelian
varieties. In particular, if the CM type $(F, \Phi)$
corresponds to an abelian variety $B$ with full CM,
then $B$ and its CM are defined over $K$,
since the Shimura variety is zero-dimensional in this case.

Now we can state our result for abelian varieties of arbitrary
dimension $g$.
Roughly speaking, we show that there exists an effectively computable
finite set $S_{K, g}$ such that for
$\ell \notin S_{K, g}$
and $\psi$ a mod-$\ell$ associated character of a $g$-dimensional
abelian variety, $\psi^a$ is equal (mod $\ell$)
to a character $\psi_{F,\Phi}^b$ 
corresponding to a CM type $(F,\Phi)$ with separating field
$K_{F,\Phi}$ contained in $K$, for some exponents $a > 0$ and $b \geq 0$.
The exponents $a$ and $b$, and the dimension
$\dim_K\Phi$ of abelian varieties of CM type $(F,\Phi)$
are both bounded by a constant depending only on $g$ and 
satisfy some additional restrictions.

In equation~\eqref{iellcurvealign},
the exponents $a$ and $b$ satisfy
$a = b = 12$, and also we have $\dim E = \dim E' = 1$.
This makes it tempting to expect that in general $a = b = c_g$, and
the CM abelian varieties representing the CM type
$(F,\Phi)$ have dimension bounded by $g$
(i.e.\ $\dim_K \Phi \leq g$). While we believe such a
result should be true, we do not know how to prove
it in general, even assuming GRH. However, we 
can prove a modified version of this statement
under some additional assumptions; see Corollary~\ref{cor:ssone}.
We also can improve our bound on $\dim_K \Phi$
and show that $\dim_K \Phi \leq g$ for $d = \deg \psi = 1$,
using the following observation.
If $\psi_0$ is an associated
character of $A$, then $\cyc_\ell^d \otimes \psi_0^{-1}$
is also an associated character of $A$, due to the
Galois-invariance of the Weil pairing.
Thus, to describe all of the associated characters
of $A$, it suffices to consider one character
in each pair $\{\psi_0, \cyc_\ell^d \otimes \psi_0^{-1}\}$.
Now we state the main theorem of our paper.

\begin{ithm}[Theorem~\ref{mainthm}]\label{intromainthm}
Let $K$ be a number field, and $g$ and $d$ be positive integers.
Then, there exists a
finite set $S_{K, g}$ of prime numbers depending only on $K$
and $g$, and a constant $0 < c_g < 12^{4g^2}$ depending only on $g$
such that, for a prime $\ell \notin S_{K, g}$,
and a $g$-dimensional abelian variety $A$
with a mod-$\ell$ associated character $\psi_0$
of {\weight} $d$,
we have
\[\psi^{e \cdot w} \equiv \psi_{F, \Phi}^w \pmod \ell,\]
where $\psi$ is either $\psi_0$ or $\cyc_\ell^d \otimes \psi_0^{-1}$
and $w = \frac{\lcm(N, c_g)}{\gcd(e, c_g)}$.
Here, $F$ is either $\qq$ or a CM-field,
and $\Phi \colon F \to \mend(K^m)$
is a primitive balanced representation
such that $K \supset K_{F, \Phi}$.
The quantities $a$, $e$, and $N$ are integers with
$e$ and $N$ positive, which
satisfy $m = \frac{1}{2} \cdot a \cdot e \cdot [F : \qq]$.
Moreover, $0 \leq a \leq d$, and both
$\phi(N)$ and $e \cdot [F : \qq]$ are
at most $\binom{2g}{d}$.
\end{ithm}

We think of $\Phi$ as giving the isomorphism class of the
$K$-$F$ bimodule $\lie(B)$,
for any abelian variety $B$ with CM type $(F, \Phi)$.
The above bounds imply that
$m = \dim B \leq \frac{d}{2} \cdot \binom{2g}{d}$.
In particular, if $d = 1$, then $m = \dim B \leq g = \dim A$.

The proof of Theorem~\ref{intromainthm} implies that the set
$S_{K, g}$ is effectively computable. In Theorem~\ref{effectivethm},
we give an explicit bound on the value of $\prod_{\ell \in S_{K, g}} \ell$.

The relationship between mod-$\ell$ and $\ell$-adic characters becomes
particularly nice when $K$ has semistable reduction at all primes over $\ell$
and the associated character has {\weight}~$1$ (i.e.\ is a one-dimensional
subquotient). In this case, we get a very similar result to the conditional
statement of Theorem \ref{introellcurve}.

\begin{icor}[Corollary~\ref{cor:ssone}]
Let $K$ be a number field, and $g$ and $d$ be positive integers.
Then, there exists a
finite set $S_{K, g}$ of prime numbers depending only on $K$
and $g$, and a constant $0 < c_g < 12^{4g^2}$ depending only on $g$
such that, for a prime $\ell \notin S_{K, g}$,
and a $g$-dimensional abelian variety $A$
with a mod-$\ell$ associated character $\psi$
of {\weight} $1$,
one of the following holds.
\begin{enumerate}
\item The character $\psi^{c_g}$ is trivial or equal to $\cyc_\ell^{c_g}$.
\item There exists an abelian unramified extension $M/K$,
a (full) CM abelian variety $A'$ defined over $M$,
such that $K$ contains the reflex field of the CM field of $A'$
(which in particular implies that $A'$ has CM defined over $M$),
and an $\ell$-adic associated character of \weight~$1$
of $A'$, whose mod-$\ell$ reduction $\psi'$ satisfies
\[(\psi|_{\gal(\bar{K}/M)})^{c_g} = (\psi')^{c_g} \quad \text{and} \quad 
(\dim A') \cdot (\exponent M/K) \leq g.\]
\end{enumerate}
\end{icor}

When $A$ is an abelian surface and $d = 1$, the result
of Theorem~\ref{intromainthm} can be formulated more concisely.

\begin{icor}[Corollary~\ref{absurf}]
Let $K$ be a number field.
Then there exists a
finite set $S_K$ of prime numbers depending only on $K$
such that, for a prime $\ell \notin S_K$,
and an abelian surface $A$
with a mod-$\ell$ associated character $\psi$
of {\weight} $1$,
one of the following holds.
\begin{enumerate}
\item There exists a full CM abelian surface $A'$
over $K$ whose CM is defined over $K$, with an
$\ell$-adic \weight~$1$ associated character
whose mod-$\ell$ reduction $\psi'$ satisfies
\[\psi^{120} = (\psi')^{120}.\]
\item
There exists an abelian unramified extension $L/K$ of exponent
at most $2$, a CM
elliptic curve $E'$ defined over $L$,
such that $K$ contains the CM field, and
an $\ell$-adic \weight~$1$ associated character of $E'$ whose mod-$\ell$
reduction $\psi'$ satisfies
\[\psi|_{\gal(\bar{K}/L)}^{120} = (\psi')^{120}.\]
\item For some $a \in \{0, 60, 120\}$, we have
\[\psi^{120} = \cyc_\ell^a.\]
\end{enumerate}
\end{icor}

There is a simlar formulation for abelian threefolds
(see Corollary~\ref{abthree}),
but not for higher-dimensional abelian varieties;
this corresponds to the fact that for $g \geq 4$, the relevant Shimura
varieties may have nonzero dimension.

If $A = \alb(X)$ is the Albanese variety of a
smooth proper scheme $X$ of finite type over $K$,
then we can apply Theorem~\ref{intromainthm} to $A$.
For $\ell$ sufficiently large,
Theorem~\ref{intromainthm} gives nontrivial
restrictions on the action of $G_K$ on the
\'etale cohomology $H^1(X, \ff_\ell) \simeq H^1(\alb(X), \ff_\ell)$.
We believe it may be possible to study the associated characters of the
higher \'etale cohomology
groups $H^r(X, \ff_\ell)$ using our methods.
Indeed, in several cases --- e.g.\ when $X$ has everywhere
good reduction --- our techniques give an analogous
result for $H^r(X, \ff_\ell)$. However, in the general case
there are some difficulties that occur due to the lack
of a good theory of semistable reduction for arbitrary smooth proper
schemes
(see Remark~\ref{general} for details).

Our proof of Theorem~\ref{intromainthm} is similar in spirit
to the method used by Serre in \cite{serre}
to classify elliptic curves with non-surjective mod-$\ell$
Galois action.
Working with abelian varieties instead
of elliptic curves introduces additional issues.
This is because while there are only finitely many
elliptic curves with a given endomorphism ring larger than $\zz$,
there are families of abelian varieties with extra endomorphisms
parameterized by positive-dimensional Shimura varieties.
This forces the theory of associated characters to become more complicated,
once these Shimura varieties enter into the picture.
Moreover, obtaining uniform bounds
requires a more delicate analysis than Serre's \cite{serre},
in particular because we cannot
assume that $A$ is everywhere semistable by extending
the ground field.

The paper consists of two main essentially independent parts.
In the first part, we explicitly compute the $\ell$-adic associated
characters of abelian varieties in terms of class
field theory. This analysis
is relatively simple, but
requires a number of technical results
including Faltings' Theorem \cite{faltings} and
the theory of Shimura varieties
of PEL type \cite{travshim}.
It also uses the local characterization of determinants
of subquotients of the $p$-adic Tate module,
proved by Brian Conrad in Appendix~\ref{detcrys} using
$p$-adic Hodge theory.
The second part of the paper, about mod-$\ell$
associated characters, is significantly
more involved but uses less machinery.
We use a result of Raynaud to get control 
over the action of the inertia groups at primes dividing
$\ell$, and use the Weil conjectures
(and Grothendieck's generalization in~\cite{monodromy}
for primes of bad reduction)
to control the behavior
of Frobenius elements at primes not dividing $\ell$.
These pieces of information are then combined
to produce a list of explicit expressions for the mod-$\ell$
associated characters in terms
of class field theory. This list turns out to be remarkably
similar to the list of $\ell$-adic characters found in the first part
of the paper (see Theorem~\ref{intromainthm}).

We now describe the structure of the paper.
In Section~\ref{sec:algchars}, we introduce
the language of algebraic characters which
we will use throughout the paper.
 
In Section~\ref{sec:bigcm} we study abelian varieties $B$
with CM defined over $K$ and their associated determinantal
characters; we also compute the separating fields
of some Shimura varieties classifying $(F, \Phi)$ abelian varieties.

Then we turn to studying mod-$\ell$ associated characters.
We begin in Section~\ref{sec:local} by studying the local behavior of
these characters. In
\ref{sec:frobenii} and~\ref{sec:semistable},
we recall some facts about the action of Frobenius elements on the Tate module
and about semistable reduction of abelian varieties respectively.
In \ref{sec:inertia}, we analyze
the restriction of $\psi$ to the inertia subgroups.

In Section~\ref{sec:global}, we study the global
properties of these characters, using the 
results from Section~\ref{sec:local}. In particular, the
analysis in Section~\ref{sec:global} will apply to any
Galois character that satisfies the key results from
Section~\ref{sec:local}; for details see Remark~\ref{general}.
In~\ref{sec:charclass},
we prove several restrictions on the character $\psi$
using its values on Frobenius elements.
Subsection~\ref{sec:varproof} is devoted to the proof
of the Theorem \ref{intromainthm}, our main result.
Finally, in~\ref{subsec:cors},
we derive some corollaries of the main theorem.

In Section~\ref{sec:curves}, we apply the main theorem
to elliptic curves and prove Theorem~\ref{introellcurve}
and Corollary~\ref{isogeny} about isogenies of elliptic curves.

Finally, in Section~\ref{sec:effective}, we make all of our arguments
effective, proving effective versions of the main theorem
both for elliptic curves and abelian
varieties (Theorems~\ref{thmcg} and~\ref{effectivethm}).

\medskip

\emph{Notation Conventions:} Throughout the paper, we normalize the Artin map to carry
uniformizers to \emph{arithmetic} Frobenius elements, i.e.\ the map
$\alpha\mapsto \alpha^q$.
We write
$n_K$, $r_K$, $R_K$, $h_K$, and $\Delta_K$ for the degree,
rank of the unit group, regulator, class number, and discriminant of $K$
respectively.

\subsection*{Acknowledgements}
We would like to thank David Zureick-Brown, Bryden Cais and Ken Ono for giving
us the problem that led to this paper and answering questions.
We are grateful to Brian Conrad for writing
Appendix~\ref{detcrys}.
Thanks also to Brian Conrad, Jordan Ellenberg, Benedict Gross, Mark Kisin, Barry
Mazur, Jean-Pierre Serre and Vadim Vologodsky 
for valuable comments and
discussions. Finally we would like to express our admiration for
Serre's paper
\cite{serre} whose techniques provide the mathematical inspiration and
departure-point of our work.

\section{\label{sec:algchars} Algebraic Characters}

A multiplicative map of number fields of algebraic origin --- for example
the norm map $\nm \colon K^\times\to \qq^\times$ --- can be thought
of as the induced map on $\qq$-points of a map of commutative algebraic
groups over $\qq$.
Such maps --- which we will call \emph{algebraic characters} ---
will appear for us as determinants of representations induced
by CM (in Section \ref{sec:cm}) and as Galois characters (see Definition
\ref{def:ext}).
In this section, we will introduce two main ways of thinking
about algebraic characters --- a concrete description as
products of embeddings (Proposition \ref{torusmaps}) and
an equivalent picture as determinants of representations
(Proposition~\ref{charsfromdets}).

The main technical result of this section, Lemma \ref{galoischars},
will give a way of
defining global \emph{Galois} characters from algebraic characters
between global fields --- an idea originally due to Serre
(in \cite{abelianladic}) but done here from the slightly 
different vantage point of CM fields and groups of Weil numbers.

\subsection{Definitions and first properties of algebraic characters}

We will be interested in algebraic characters between pairs of global
number
fields and pairs of $p$-adic number fields,
but will first give some general definitions over any base field $Q$.
Suppose that $K$ and $L$ are algebraic extensions of a base field $Q$
and $K/Q$ is finite.
Write $T_K$ and $T_L$ for the algebraic tori 
$\res^K_Q \gm{K}$ and $\res^L_Q \gm{L}$, viewed
as algebraic groups over $Q$.
\begin{defi}
  An \emph{algebraic character} 
is a multiplicative map $\theta\colon K^\times\to L^\times$ 
given by the map on $Q$-points of a morphism of
algebraic groups $T_K\to T_L$. 
\end{defi}

\noindent
Now pick an embedding $L\hookrightarrow \bar{Q}$.

\begin{defi} We write $\emb{K}$ for the set of embeddings
$K \hookrightarrow \bar{Q}$.
\end{defi}

\begin{prop}\label{torusmaps}
The (discrete) abelian group of multiplicative maps $\hom(T_K,T_L)$
is the group of invariants $\zz[\emb{K}]^{G_L}$, with
induced map on $Q$-points 
\[\sum_{\sigma\in\emb{K}}a_\sigma\sigma\colon x\mapsto \prod (x^\sigma)^{a_\sigma} \quad \text{for} \quad x\in T_K(Q)=K^\times.\]
\end{prop}
\begin{proof}
See \cite{abelianladic}, section~1.1 of chapter~2.
\end{proof}

\begin{cor}
If $L$ contains the Galois closure $\kgal \subset \bar{Q}$ of $K$,
then
\[\hom(T_K,T_L)=\zz[\emb{K}].\]
\end{cor}

\begin{rem}
For the purposes of this paper, one can always think
of algebraic characters as maps $K^\times \to L^\times$
of form given by Proposition \ref{torusmaps}. 
\end{rem}

\begin{defi}
For $S = \sum S(\sigma) \cdot \sigma \in \zz[\emb{K}]$,
we define the algebraic character $\theta^S$ via
\[\theta^S(x) = \prod \sigma(x)^{S(\sigma)}.\]
\end{defi}

\begin{defi}
We say that a character $\theta\colon T_K\to T_L$ is \emph{positive}
if it extends
from $T_K$ to $\res^K_Q \mathbb{A}^1_K$, or equivalently if $\theta=\theta^S$
for $S=\sum S(\sigma)\cdot\sigma$ with all $S(\sigma)$ nonnegative.
Similarly, we say that $\theta$ is \emph{negative} if
$\theta \circ (x \mapsto x^{-1})$ is positive.
\end{defi}

\begin{defi} We define the \emph{degree} of an algebraic
character $\theta^S$ to be $\max_{\sigma \in \emb{K}} |S(\sigma)|$.
\end{defi}

Now we give the dictionary between characters and
representations.
  Suppose $V$ is a $K\otimes_Q L$ bimodule which is finite-dimensional
as an $L$-module. Thinking of $V$ as an $L$-vector space
with an action of the (algebraic) group $K^\times,$ 
we define $\det_LV \colon K^\times\to L^\times$
via the $L$-determinant of the $K$-action, i.e.
\[a \mapsto \det_L\big((x \mapsto a \cdot x) \colon V \to V\big).\]
\begin{prop}\label{charsfromdets}
For any positive algebraic character $\theta\colon K^\times\to L^\times$,
there is
a $K\otimes_QL$-bimodule $V$ which is finite-dimensional
as an $L$-module and such that $\det_LV=\theta.$ Moreover,
the bimodule $V$ is unique up to isomorphism.
\end{prop}
\begin{proof}
Suppose at first that $L = \bar{Q}$.
Let the $a_\sigma$ be as in Proposition~\ref{torusmaps}.
Then the representation $V = \bigoplus_{\sigma \in \emb{K}} a_\sigma \sigma$
is the required bimodule. This bimodule is unique since
$K \otimes_Q \bar{Q} \simeq \bigoplus_{\sigma \in \Gamma_K} \bar{Q}[\sigma]$,
where $\bar{Q}[\sigma]$ is a one-dimensional vector space
over $\bar{Q}$ such that the $K$-action is given by $\sigma$
composed with the $\bar{Q}$-action.

For the general case, we invoke the primitive element
theorem to write $K = Q[\alpha]$;
note that our representation is determined by the matrix for $\alpha$
up to conjugacy.
The result thus follows from the following standard
result of linear algebra applied to $N / M = \bar{Q} / L$,
where the Galois-invariance is evident in the 
proof of the case $L = \bar{Q}$:
Given a field extension $N / M$,
the intersection of any Galois-invariant conjugacy class of $\gl_n(N)$ with
$\gl_n(M)$ is a conjugacy class of $\gl_n(M)$.
(This last fact follows from the invariance of the rank
of a matrix under base extension.)
\end{proof}

\begin{rem} When $K$ and $L$ are both
finite-dimensional over $Q$,
the above proposition gives a curious duality between
algebraic characters $K^\times\to L^\times$ and 
$L^\times\to K^\times$, known as the reflex map.
\end{rem}

Now suppose that $\oo \subset Q$ is an integrally closed subring.
Since embeddings of fields preserve
rings of integers, Proposition \ref{torusmaps} implies that
any algebraic character
$\theta\colon K^\times \to L^\times$ sends $\oo_K^\times\subset K^\times$ to 
$\oo_L^\times\subset L^\times$.

\begin{defi}
For an ideal $I\subset \oo_L$ and
algebraic character
$\theta\colon K^\times \to L^\times$ 
we define 
\[\theta\ \text{mod} \ I \colon \oo_K^\times\to (\oo_L/I)^\times\]
to be the map induced by composing with $\oo_L^\times \to (\oo_L/I)^\times$.
We call a character of the above form a \emph{reduction of an algebraic
character}.
\end{defi}

We will be particularly interested in the case where $Q=\qq_p$ and
$I=m_L$ is the maximal ideal of $\oo_L$, giving a map $\theta \
\text{mod} \ m_L \colon \oo_K^\times \to k_L^\times$.

\subsection{\boldmath $\ell$-adic characters induced from global characters}
\label{sec:pinduced}

Suppose $K/\qq$ is a number field and $\ell$ is a prime.
We will take $L=\qqbar$, and fix an embedding
$\iota\colon\qqbar\hookrightarrow \bar{\qq}_\ell$.
\begin{prop}\label{prop:loc}
There is a (unique)
bijective correspondence
\[ (\theta\mapsto\theta_\ell) \colon \hom\left(T_K,T_{\qqbar}\right) \overset{\sim}{\longrightarrow} \hom\left(\prod_{\pp \mid \ell} T_{K_\pp},T_{\bar{\qq}_\ell}\right),\]
with the property that the maps on points $K^\times \to \bar{\qq}^\times$
and $\prod K_\pp^\times\to \bar{\qq}_p^\times$ fit into
the commutative diagram
\[\begin{CD}
K^\times @>\theta>> \qqbar \\
@VVV @VV{\iota}V \\
\prod_{v \mid \ell} K_v @>>{\theta_\ell}> \qqbar_\ell
\end{CD}\]
\end{prop}
\begin{defi} \label{def:loc}
We will call this bijective correspondence \emph{localization}
of characters.
\end{defi}
\begin{proof}
  Note that if $\theta_\ell$ exists, it is unique since
$K^\times\subset \prod_{\pp\mid \ell} K_v$ is dense. Now suppose
$\theta=\sigma$ is a single embedding. Then by the decomposition
theorem for ideals under of finite extensions, the composition
$\iota\circ\sigma$ can be uniquely written as 
$K\subset K_\pp\overset{\sigma_\ell}{\longrightarrow}\bar{\qq}_\ell$, where 
$\sigma_\ell \in\emb{K_\pp}$. Having constructed
$\sigma_\ell$ for $\sigma\in\emb{K_\pp}$ which fits into the 
diagram in the proposition, we define
$\theta_\ell$ for arbitrary $\theta$ multiplicatively.
This is an isomorphism $\hom(T_K,T_{\qqbar})
\overset{\sim}{\longrightarrow} \hom\big(\prod_{\pp\mid p} T_{K_\pp},T_{\bar{\qq}_\ell}\big)$
as it induces a bijection
between generators of two free groups.
\end{proof}

\subsection{\label{sec:galchars}
Galois characters induced from algebraic characters}
Localizing an algebraic character $\theta\colon K^\times\to L^\times$ 
of global fields, one can define a number of characters
on the group of id\`eles (for example,
by considering the induced map on
adelic points $T_K(\af) \to T_L(\af)$),
which will in some cases descend to Galois characters. Here
we give such a construction in a rather specific case,
which will be of interest in the rest of the paper.

\begin{defi} \label{def:good}
We say that an algebraic character
$\theta\colon K^\times\to L^\times$ is \emph{balanced} if
the product $\theta \cdot (\theta\circ\sigma)$
is a power of $\nm^K_\qq$ for \emph{any} complex conjugation 
$\sigma\in\gal(L/\qq)$. 
Otherwise, we say $\theta$ is \emph{unbalanced}.
Similarly, we say an $K$-$L$ bimodule $V$ is balanced (resp.\ unbalanced)
if the character $\det_L V$ is balanced (resp.\ unbalanced).
\end{defi}

The following two lemmas, which will play a central role in
the following sections, show that if we hope to extend an algebraic
character $\theta$ (thought of as a Galois character of a product
of local field)
to a \emph{global} Galois group, then (under certain mild
ramification conditions) $\theta$ has to be balanced.

\begin{lm}\label{restons}
The character $\theta^S$ is balanced if and only if $\theta^S(u)$
is a root of unity for any unit $u \in \oo_K^\times$.
\end{lm}
\begin{proof} 
Let us order the $\sigma_i$ so that
$\sigma_1, \sigma_2, \ldots, \sigma_{r_1}$
are real embeddings,
and $\sigma_{r_1 + i}$ is the complex conjugate
of $\sigma_{r_1 + r_2 + i}$.
Consider the multiplicative 
Minkowski embedding, i.e.\ the map
$\mu\colon\oo_K^\times \to  \rr^{r_K + 1}$
given by
\[\mu \colon x \mapsto \big(\log |\sigma_1(x)|, \log |\sigma_2(x)|, 
\ldots, \log |\sigma_{r_1}(x)|, 2\log |\sigma_{r_1 + 1}(x)|, 
\ldots, 2\log |\sigma_{r_1 + r_2}(x)|\big).\]
Dirichlet's unit Theorem states that the kernel of $\mu$ is roots of unity
and
the image of $\mu$ is a lattice in ${\rr^{r_K+1}}_0,$
the subspace of vectors in $\rr^{r_K+1}$ whose sum of coordinates is zero.
Call this lattice $\Lambda=\im(\mu)\subset {\rr^{r_K+1}}_0$. 
Observe that for any $S \in \zz[\Gamma_K]$, the function
\[\left(x \mapsto \log \left|\theta^S(x)\right|\right) = f^S \circ \mu\]
factors as the composition of $\mu$ with a linear function
$f^S \colon \rr^{r_K+1} \to \rr$.
Since $\theta^S(u)$ is a root of unity for any unit $u \in \oo_K^\times$,
it follows that $f^S$ vanishes on units;
moreover, since $f^S$ is linear, $f^S$ vanishes on the hyperplane spanned
by the units. Hence it must be a multiple of the defining equation
of the hyperplane, 
$\sum_{i=1}^{r_1 + r_2} x_i$.

The converse follows from the fact that if $|\theta(u)|=1$ under
any complex embedding, then $\theta(u)$ is a root of unity.
\end{proof}

\begin{lm} \label{ladicgood}
If the localization $\theta^S_\ell$ coincides on an open subgroup 
$U\subset \prod_{v\mid \ell}\oo_{K_v}^\times$ 
with the composition
\[\begin{CD}
\prod_{v \mid \ell} \oo_{K_v}^\times @>>> \ii_K @>>> \gal(K^{\text{ab}} / K) @>\psi>> \bar{\qq}_\ell^\times
\end{CD}\]
for a character $\psi \colon \gal(K^{\text{ab}} / K) \to \bar{\qq}_\ell^\times$
which is ramified at only finitely many primes
and has finite ramification index at primes not lying over $\ell$,
then $\theta^S$ is balanced.
\end{lm}
\begin{proof}
Let $(\oo_K^\times)_\ell$ be the image
of the group of units under the
embedding $\oo_K^\times \hookrightarrow \prod_{v\mid \ell} \oo_{K_v}^\times$,
and let $U_\ell\subset \ii_K$ be the subgroup 
$(\oo_K^\times)_\ell \times \prod_{v\nmid \ell}\oo_{K_v}^\times$.
Then the image of $U_\ell$ in $\ii_K/K^\times$ coincides
with the image of $\prod_{v\nmid\ell} \oo_{K_v}^\times$,
which (by the finite ramification outside of $\ell$ assumption) has 
finite image under $\psi$. On the other hand, $\psi(U_\ell)$
contains
a subgroup of finite index of the image 
$\theta^S(\oo_K^\times)$. Hence, the image of $\oo_K^\times$
under $\theta^S$ is finite, so
$\theta^S$ is balanced by Lemma~\ref{restons}.
\end{proof}

Next we prove a converse --- that if the character \emph{is}
balanced then we can in fact extend it to a global Galois character
with appropriately mild ramification behavior.

\begin{defi}
The \emph{group of Weil elements} $\weil{L} \subset L^\times$ is the group
of elements $w\in L^\times$ satisfying $w\cdot \sigma(w) \in \qq^\times$
for \emph{any} complex conjugation $\sigma\in \gal(\kgal/\qq)$.
\end{defi}

\begin{defi}
  We say that $F$ is a \emph{CM field} if \emph{any}
complex conjugation restricts
to the same nontrivial element of $\gal(F^{\text{gal}}/\qq)$.
Equivalently, $F$ is a CM field if $F$ is a quadratic totally imaginary
extension of a totally real subfield.
\end{defi}

It is easy to see that
a balanced character $K^\times\to L^\times$ factors through an embedding 
$F\subset L$, where $F$ is either a CM field or $F = \qq$.
We will from now on take $F$ to be a CM field (or $\qq$)
and $\theta\colon K\to F$ to be an algebraic character.

\begin{defi} \label{weilclass}
The \emph{Weil class group} $\cla^W(F)=I_F/W_F$ is the
group of fractional ideals modulo the group of Weil elements of $F$.
(Note that it is in general an infinite group.)
\end{defi}

As $\theta$ is a map of algebraic groups, $\theta$
induces a map from the id\`eles of $K$
to the id\`eles of $L$. Since $\theta$ is balanced,
the image (viewed as a map $K^\times \to F^\times$)
of $\theta$ lies in the group of Weil elements of $F$.

\begin{defi} We define $C\theta \colon \cla(K) \to \cla^W(F)$
to be the map induced by $\theta$.
\end{defi}

\begin{defi}\label{def:ktheta} Write $\ii_K^\theta$ for the group of id\`eles of $K$
whose ideal class is in the kernel of $C\theta$.
Write $K_\theta$ for the abelian extension of $K$
corresponding to the subgroup $\ii_K^\theta \subset \ii_K$.
\end{defi}

\begin{defi} We write $N_0$ for the number of roots
of unity in $F$, and $\mu_F = \mu_{N_0}$ for the group
of roots of unity in $F$.
\end{defi}

A Weil element which
is a unit has norm $1$ under any complex embedding, and hence is
a root of unity.
Therefore, $\theta$ induces a map $I\theta$ from $\ii_\theta$ to
$W_F/\mu_F$.
Now, we fix an embedding $F\hookrightarrow \qqbar_\ell$.

\begin{defi} \label{def:ext}
For a map $\theta:K^\times\to F^\times\subset \qqbar^\times$, 
we define the character
$\psi_\theta \colon \ii_K^\theta \to \qqbar_\ell / \mu_F$
via $\psi_\theta := I\theta \cdot \theta_\ell^{-1}$.

Slightly abusing notation, given a $K-F$ bimodule $\Phi$
we define $\psi_{F,\Phi}:=\psi_\theta$ for $\theta=\det_K\Phi:K^\times\to \qqbar^\times$.
\end{defi}

\begin{lm}\label{galoischars}
The character $\psi_\theta$ is trivial on principal id\`eles
and has image contained in $\oo_{\qqbar_\ell}^\times/\mu_F$.
\end{lm}
\begin{proof}
  To show triviality on principal id\`eles, it
suffices to note that given $x\in K^\times,$ we have
an equality of principal ideals $(I\theta(x))=(\theta(x))$,
both of which are generated by a Weil number. 
To show that the image lies in $\oo_{\qqbar_\ell}^\times,$ it suffices
to note that $\psi_\theta(\pi_\pp)$ and $\theta_\ell(\pi_\pp)$ have
the same norm for $\pi_\pp\in K_\pp$ a uniformizer.
\end{proof}
Extending the character to infinite places by the trivial map,
this lets us define a Galois character
(on $\gal(\bar{K} / K_\theta)$)
given any CM field $F$
and balanced character $K^\times \to F^\times$.
\begin{rem}
  The above character is uniquely determined by the property
that it takes any uniformizer $\pi_\pq$ for $\pq\nmid\ell$
and with $[\pq]\in \ker C\theta$
to a Weil number generating the ideal $I\theta(\pq)$
(up to roots of unity).
\end{rem}

\section{\label{sec:bigcm}Abelian Varieties with Complex Multiplication}
\subsection{\label{sec:cm} 
Galois Characters Associated to CM Abelian Varieties}
Let $B$ be an abelian variety defined over a field $K$ of
characteristic zero.
Recall that we say $B$ has endomorphisms by a
number field $E$ if there is an injection
$\iota \colon E \hookrightarrow \mend_K(B) \otimes \qq$.

The $p$-adic Tate module is a finite-dimensional
$(E \otimes \qq_p)$-module, so we can take the determinant over
$(E \otimes \qq_p)$ of the Galois action.
(If we write $E \otimes \qq_p = \prod_{v \mid p} E_v$,
this is just the product of the $E_v$-linear determinants.)
This determinant of the Galois action gives us a map
$\gal(\bar{K} / K) \to (E \otimes \qq_p)^\times$.

Note that the $E$-action gives a decomposition of the Tate
module with extended coefficients, $B[p^\infty]\otimes \bar{\qq}_p$,
into eigenspaces 
(corresponding to embeddings $E\hookrightarrow \bar{\qq}_p$).
The $E\otimes\qq_p$-determinant encodes the 
$\bar{\qq}_p$-determinants (i.e.\ associated characters) of
each of these eigenspaces.
The $\bar{\qq}_p$-determinant
of the $\sigma$-eigenspace of $E$ is given by the character
$\psi_{B, \sigma}$ defined as follows.
\begin{defi}\label{detsigma}
Define the character \[\psi_{B,\sigma}:=(\sigma\otimes\id)\circ
\det_{E \otimes \bar{\qq}_p}B[p^\infty]\otimes\bar{\qq}_p.\]
\end{defi}

We will see
that the study of $\ell$-adic associated characters 
can be reduced to studying the characters $\psi_{B,\sigma}$.
Namely, suppose that $B$ is an abelian variety.
Let $V$ and $W$ be subspaces of $B[\ell^\infty] \otimes \bar{\qq}_\ell$;
write $V^0$ and $W^0$ for $V \cap B[\ell^\infty] \otimes\bar{\zz}_\ell$
and $W \cap B[\ell^\infty] \otimes \bar{\zz}_\ell$ respectively.
We define the distance function
$d(V, W) = \ell^{-n(V, W)}$, where
\[n(V, W) = \max \left\{m : V^0 + \ell^m B[\ell^\infty]\otimes\bar{\zz}_\ell = 
W^0 + \ell^m B[\ell^\infty]\otimes\bar{\zz}_\ell\right\}.\]
Then we have the following lemma.

\begin{lm}\label{splitimpliescm}
Suppose $B$ is an abelian variety over a number field $K$.
Let $V$ be an irreducible
$G_K$-equivariant subspace (or, in this case equivalently,
subquotient) of $B[\ell^\infty]\otimes\qqbar_\ell$.
Then for any $\epsilon > 0$,
there is an $\alpha \in \mend(B)$
such that $d(V, W) \leq \epsilon$
for some eigenspace $W$ of the action of $\alpha$
on $B[\ell^\infty]\otimes\qqbar_\ell$.
\end{lm}
\begin{proof}
Let $V=\sigma_0 V,\sigma_1V,\dots, \sigma_mV$ 
be the conjugates of $V$ under the action
of $\gal(\bar{\qq}_\ell/\qq_\ell)$.
By a theorem of Faltings (Theorem~3 of \cite{faltings}),
the Galois representation
$B[\ell^\infty]\otimes\qq_\ell$ is semisimple. This means it
remains semisimple after extending coefficients to $\qqbar_\ell,$
and so $B[\ell^\infty]\otimes\qqbar_\ell=
\sigma_0V\oplus\dots \oplus \sigma_mV\oplus V'$ for some
subrepresentation $V'$. Now let $H$ be the subgroup
of $G_{\qq_\ell}$ of elements which take $V$
to itself, and let $\kk\subset \qqbar_\ell$ be the
fixed subfield of $H$. The extension $\kk/\qq_\ell$ is finite,
so we can choose $\alpha\in \kk$ a primitive element
over $\qq_\ell$.
Now we define an endomorphism 
\[\alpha_0 = \sigma_0(\alpha)\cdot \id_{\sigma_0 V}\oplus\dots\oplus
\sigma_m(\alpha)\cdot \id_{\sigma_m V}\oplus \, 0_{V'} \colon
B[\ell^\infty]\otimes \qqbar_\ell\to B[\ell^\infty]\otimes\qqbar_\ell,\]
and notice that it is $\gal(\kk / \qq_\ell)$-invariant, and hence
$\alpha_0$
restricts to a
$G_K$-invariant
endomorphism
$\alpha_0 \colon B[\ell^\infty] \otimes \qq_\ell \to B[\ell^\infty] \otimes \qq_\ell$.

Another theorem of Faltings (Theorem~4 of \cite{faltings})
gives that
\[\mend(B)\otimes\qq_\ell\cong 
\mend_{G_K}(B[\ell^\infty]\otimes \qq_\ell).\] 
In particular, there is an element $\alpha\in\mend B\otimes\qq$
whose induced map on the Tate module
approximates $\alpha_0$ to arbitrary $\ell$-adic precision.
Then $V$ is arbitrarily close to a single eigenspace of 
the action of $\alpha$ on $B[\ell^\infty]\otimes \qqbar_\ell$.
\end{proof}

Now suppose $B$ is simple
(we can study the general case by decomposing $B$ up to
isogeny into a product of simple abelian varieties).
We take $E = \qq(\alpha)$, so $\det_W = \psi_{B, \sigma}$.
In the next two sections, we will see that up to multiplication
by roots of unity in $E$, the character $\psi_{B, \sigma}$
is determined by the field $E$ plus some \emph{finite} 
combinatorial data. In particular (up to multiplication
by bounded roots of unity), there are only finitely many 
characters giving the Galois action $\det_W$ in the
above lemma; it follows by taking a sufficiently good
approximation $W$ that one of these finitely many characters
gives the action on $\det_V$.

\subsection{Local Case}
Suppose that $K=\kk$ is a $p$-adic field. By local class field theory,
we have a natural map
$\operatorname{rec} \colon \kk^\times \to \gal(\kk^\text{ab} / \kk)$.

Now, consider the determinant of the Galois action
$\gal(\bar{\kk} / \kk) \to (E \otimes \qq_p)^\times$.
Since $(E \otimes \qq_p)^\times$ is an abelian group, this map
factors through $\gal(\kk^\text{ab} / \kk)$.
Precomposing with the reciprocity homomorphism, we get a map
$\kk^\times \to (E \otimes \qq_p)^\times$.
Write $\lie(B)$ for the tangent space to the identity element
of $B$. Since $\qq_p \subset \kk$ acts on $\lie(B)$, the Lie algebra
$\lie(B)$ naturally has the structure of an $(E \otimes \qq_p)$-$\kk$ bimodule.
We get another map $\kk^\times \to (E \otimes \qq_p)^\times$
induced by taking the determinant over $E\otimes \qq_p$
of the $\kk^\times$-action on the
$(E\otimes\qq_p)$-$\kk$ 
bimodule $\lie(B)$ (c.f.\ Proposition~\ref{charsfromdets}).
These maps do not necessarily coincide, but they almost do
(up to sign);
namely, we have the following theorem, due to Conrad.

\begin{cthm}[Conrad, Appendix~\ref{detcrys}] \label{lmlie}
Let $B$ be an abelian variety defined over a local field $\kk$
of residue characteristic $p$, which admits an injection
$\iota \colon E \hookrightarrow \mend_{\kk}(B) \otimes \qq$.
Then, there is an open subgroup $U \subset \oo_{\kk}^\times$ on which
the following diagram commutes.
\[\begin{CD}
U \subset \kk^\times @>{\operatorname{rec}}>> \gal(\kk^\text{ab} / \kk) \\
@V{u \mapsto (x \mapsto u^{-1} \cdot x)}VV @VV{\det_{E \otimes \qq_p} \circ \rho_{B, p^\infty}}V \\
\aut_{E \otimes \qq_p}(\lie(B)) @>>{\det_{E \otimes \qq_p}}> (E \otimes \qq_p)^\times
\end{CD}\]
\end{cthm}

\begin{rem} When $B$ has semistable reduction,
we can take $U = \oo_{\kk}^\times$ (see Appendix~\ref{detcrys}).
By Lemma
\ref{inertiabound}, this implies that we can 
take $U$ to have index dividing the constant $c_g$
defined in Section~\ref{sec:semistable} below.
\end{rem}

The above lemma implies that, choosing an embedding
$\sigma\colon E\hookrightarrow\bar{\qq}_p$,
the character $\psi_{B,\sigma}$ (Definition \ref{detsigma})
is the same as $(\det_E\lie(B))_p$
(see Propositions~\ref{charsfromdets} and~\ref{prop:loc})
when restricted to an open subgroup.

\subsection{Global Case}
Now suppose $K$ is a global field and
$B$ is a simple abelian variety over $K$.
As explained in Section~\ref{sec:cm},
any $\ell$-adic associated character of
$B$ is given by a $\psi_{B, \sigma}$, corresponding to
an eigenspace for the action of a
number field $E \hookrightarrow \mend(B) \otimes \qq$.
We see from Conrad's Theorem \ref{lmlie} above
that $\psi_{B, \sigma}$ equals $(\det_E \lie(B))_{\ell}$
(in the notation of Definition \ref{def:loc}),
when restricted to an open subgroup of
$\prod_{v\mid \ell}\oo_v^\times\subset \gal(\bar{K}/K)^{ab}$.

Define $F \subset E$
to be the composite of all CM subfields of $E$
(so $F$ is either a CM field or $F = \qq$). Then we have
the following lemma.

\begin{lm}\label{inducedfromphi}
The $E$-representation $\lie(B)$ is
induced from some $F$-representation $\Phi$,
i.e.\ $\lie(B)\simeq \Phi\otimes_F E$.
Moreover, $\det_E \lie(B) = \det_F \Phi$.
\end{lm}
\begin{proof}
By Lemma~\ref{ladicgood}, $\det_E \lie(B)$ is
a balanced character, and hence has image contained in a
CM field (or $\qq$). 
Equivalently by Lemma~\ref{charsfromdets}, the
$E$-module $\lie(B)$ is induced from a module over $F$, i.e.\
$\lie(B)\simeq \Phi\otimes_F E$ 
for some $F$-$K$ bimodule $\Phi$.
It follows that
$\det_E \lie(B) = \det_F \Phi$.
\end{proof}

\begin{defi}\label{def:cmtype}
We say that a polarized abelian variety $B$
is a \emph{$(F,\Phi)$-abelian variety}
if $B$ has endomorphisms by a number field $E \supset F$
such that the action of $E$ makes $\lie(B) \simeq \Phi \otimes_F E$
as an $E$-$K$ bimodule.
We call $(F, \Phi)$ the \emph{CM-type} of $B$.
\end{defi}

Lemma \ref{inducedfromphi} shows that any $\ell$-adic
associated character of a simple abelian variety
is a product of characters of the form $\psi_{B, \sigma}$ for an
$(F,\Phi)$-abelian variety $B$.

\begin{thm}\label{psifphi}
Suppose $B$ is an $(F,\Phi)$ abelian variety and $\theta:K\to \qqbar$ is
the algebraic character $\det_F \Phi$ induced from $\Phi$ under the correspondence 
of Proposition \ref{charsfromdets}. Then, in the notation of Definitions \ref{def:ktheta}
and \ref{def:ext}
we have $K = K_\theta$ and $\psi_{B, \sigma} \equiv \psi_\theta$
mod $\mu_F$.
\end{thm}

Before proving the theorem, we give an equivalent statement
which we will use in the following sections.
Let $\theta=\det_F\Phi$ as before, and write $F' \subset K$
for the minimal field such that
$\theta \colon K^\times \to F^\times$ factors as the composition of
$\nm^K_{F'}$ with $\theta_0 \colon F' \to F$.

\begin{defi} We write $K_{F, \Phi} = K_{\theta_0}$ and $\psi_{F, \Phi} = \psi_{\theta_0}$.
\end{defi}

\begin{rem}
We have $K_\theta = K \cdot K_{F, \Phi}$, and
$\psi_\theta = \psi_{\theta_0} \circ \nm^K_{F'}$.
Both the field $K_{F, \Phi}$ and the character
$\psi_{F, \Phi}$ are \emph{geometric}
invariants of $B$; i.e.\ they are unchanged under base extensions.
\end{rem}

Unwinding the definitions, we see that 
the first part of Theorem~\ref{psifphi} is equivalent
to the assertion that $K \supset K_{F, \Phi}$.

\begin{proof}[Proof of Theorem \ref{psifphi}]
First we show that $\psi_{B,\sigma} \equiv \psi_\theta$
mod $\mu_F$ restricted to $\gal(\bar{K} / K_\theta)$,
where both characters are defined.
Consider the quotient 
\[\epsilon:=\psi_{B,\sigma}/\psi_{\theta} \colon \gal(\bar{K} / K_\theta) \to \bar{\qq}_\ell/\mu_F.\]
The two characters $\psi_{B,\sigma}$ and $\psi_{\theta}$
coincide on an open subgroup of any inertia subgroup
$\oo_{K_\pp}^\times\subset \ii_K$
for $\pp\mid \ell$ and have finite
ramification degree outside of $\ell$, so the image
of $\epsilon$ is finite. Now, we use the following lemma.

\begin{lm}\label{cmfrobpoly}
For any prime $\pp$ of good reduction for $B$, there is an embedding
$E\subset\qqbar_\ell$ such that 
$\psi_{B,\sigma}(\pi_\pp)\in E\subset\qqbar_\ell$.
Furthermore, if $F \neq \qq$, then $\psi_{B,\sigma}(\pi_\pp)\in F^\times$.
\end{lm}
\begin{proof}
We have $\psi_{B,\sigma}(\pi_\pp)\in E \subset \qqbar_\ell$
by paragraph~11.10 of~\cite{symplect}.
If $F \neq \qq$, then $F$ contains all totally
real fields as well as all CM fields;
since $\psi_{B,\sigma}$ is a Weil number, it follows
that $\psi_{B, \sigma} \in F$.
\end{proof}
By definition, $\psi_\theta$
takes values in $F^\times/\mu_F$, 
so $\epsilon$ takes Frobenius elements to elements
of $\mu_E/\mu_F=\{1\}$.
Hence, by the Chebotarev density theorem, $\epsilon$ is trivial.

\medskip

Now it remains to prove that $K = K_\theta$. Although
this follows from Lemma \ref{irredthm} of the next section, we 
give an alternative proof here.
For $F = \qq$, this is immediate; thus, we assume $F \neq \qq$. 
Now suppose $\pp$ is 
a prime ideal of $K$ of good reduction for $B$.
Since $\pi_\pp^{h_K} \in \gal(\bar{K} / K_\theta)$, we have
$\psi_{B, \sigma}(\pi_\pp^{h_K}) \equiv \psi_\theta(\pi_\pp^{h_K})$
mod~$\mu_F$.
By Lemma~\ref{cmfrobpoly}, $\psi_{B, \sigma}(\pi_\pp) \in F$,
so 
$(\psi_{B, \sigma}(\pi_\pp))^{h_K} = (\psi_\theta(\pi_\pp))^{h_K}$
as ideals of $F$.
As the group of ideals of $F$ is torsion-free,
$(\psi_{B, \sigma}(\pi_\pp)) = (\psi_\theta(\pi_\pp))$.
Hence, $\theta(\pp)$ is generated by 
$\psi_{B, \sigma}(\pi_\pp)$, which is a Weil element of $F$.
By the Chebotarev density theorem,
we conclude that $C\theta \colon \cla(K) \to \cla^W(F)$
is trivial, i.e.\ $K = K_\theta$.
\end{proof}

\subsection{\label{sec:shimvar}\boldmath Shimura Varieties and the Field $K_{F, \Phi}$}
We have seen in the previous section
that if there is an $(F, \Phi)$-abelian variety $B$ defined
over a number field $K$,
then $K$ contains a certain field $K_{F,\Phi}$ associated to the CM type
$(F,\Phi)$. One can ask whether the converse holds, namely:

\begin{question}\label{cmq}
Does there exist an abelian variety with CM by $(F,\Phi)$ defined
(and with CM defined) over $K_{F,\Phi}$?
\end{question}

To the best of our knowledge, the above question is open,
although we suspect it is false in general.
However we give a statement (Theorem \ref{irredthm} below)
which is the best we can do short of answering Question \ref{cmq}.
This section will be devoted to formulating and proving this
statement, which will boil down to a computation with Shimura
varieties. 
No result from this section will be used in the rest of the paper.

Suppose $X$ is a reduced scheme of finite type over a field
$K\subset \cc$ and $B\to X$ is a (flat) family of
polarized abelian varieties over $X$.

\begin{defi}
The family $B \to X$ is a \emph{strong $(F, \Phi)$-family of abelian varieties}
if there is a map of rings
$\iota\colon \oo_F\hookrightarrow \mend(B/X)$
making every $\cc$-fiber of $B \to X$ an $(F, \Phi)$-abelian variety
with $E = F$ and
$\theta \circ \iota(\bar{\alpha}) = \iota(\alpha)^* \circ \theta$,
where $\theta \colon B \to B^\vee$ is the polarization.
\end{defi}

\noindent
We might hope to weaken Question~\ref{cmq} by replacing
$\spec K_{F, \Phi}$ with a geometrically irreducible
$K$-scheme $X$ of finite type. In other words at least morally,
for any field $K$,
there should be a strong
$(F, \Phi)$-family of abelian varieties $B \to X$
over a geometrically irreducible $K$-scheme of
finite type if and only if $K \supset K_{F, \Phi}$.
The main result of this section will be the following theorem,
which essentially states this is true if we allow
``families whose members are defined up to isomorphism.''
More precisely, we have:

\begin{thm}\label{irredthm}
Let $(F, \Phi)$ be a CM type, with $\Phi$ balanced.
Then there exists a coarse moduli space of strong $(F, \Phi)$-abelian
varieties, which is defined over $K$ provided that $\Phi$ is defined
over $K$.
Moreover, the separating field of this moduli space (i.e.\ the minimal
field over which 
any irreducible component
is geometrically irreducible)
is $K_{F, \Phi}$.
\end{thm}

To prove this theorem, we will use the language of Shimura varieties
classifying polarized abelian varieties with CM. 

First we prove that $\Phi$ must be a balanced character,
which can be reformulated as the following lemma.
\begin{lm}\label{nonempty}
There exists an $(F, \Phi)$-abelian variety
if and only if
$\Phi \oplus \bar{\Phi} \simeq F^b \otimes_{\qq} K$
for some integer $b$ (i.e.\ if and only if $\Phi$ is balanced).
\end{lm}
\begin{proof}
Suppose there was
some $(F, \Phi)$-abelian variety $B$ defined over $\cc$.
Write $B(\cc) = \cc^g / \Lambda$.
Note that $\Lambda \otimes_\zz \qq$ is an
$F$-vector space, hence isomorphic to $F^b$ for some integer $b$.
Therefore, as representations of $F$:
\[\Phi \oplus \bar{\Phi} \simeq \lie(B) \oplus \lie(B^\vee) \simeq \Lambda \otimes_\zz \cc \simeq (\Lambda \otimes_\zz \qq) \otimes_\qq \cc \simeq F^b \otimes_\qq \cc.\]

Now, we show the converse. Since
$\Phi \oplus \bar{\Phi} \simeq F^b \otimes_\qq \cc$
and all irreducible complex representations of $F$ are one-dimensional,
we can write
\[\Phi = \bigoplus_{i = 1}^b \Phi_i \qquad \text{where} \quad \Phi_i \oplus \bar{\Phi_i} \simeq F \otimes_\qq \cc.\]
By the theory of complex multiplication \cite{milne},
there exists a complex abelian variety $B_i$ having CM
by $\oo_F$ with CM type $\Phi_i$. Then
$B = \prod_{i = 1}^b B_i$ is an $(F, \Phi)$-abelian variety,
completing the proof.
\end{proof}

We now define and study the (coarse) moduli space of 
strong $(F, \Phi)$-abelian
varieties. This moduli space
has infinitely many connected components,
which we can remedy by 
keeping track of certain combinatorial data.
The first such piece of data is
the skew form $\langle \cdot, \cdot \rangle$ on the
adelic Tate module $H_1(B, \hat{\zz}) = \prod_\ell B[\ell^\infty]$
induced by the polarization.
We can also keep track of
the action $\alpha \colon \oo_F \to \mend(H_1(B, \hat{\zz}))$,
which must satisfy
\begin{equation} \label{cmform}
\langle \alpha(a) \cdot x, y \rangle = \langle x, \alpha(\bar{a}) \cdot y \rangle \quad \text{for all $a \in \oo_F$.}
\end{equation}
Furthermore, we note that as symplectic $F$-representations,
$H_1(B, \hat{\zz}) \simeq H_1(B, \zz) \otimes \hat{\zz}$ and
$\Phi \oplus \bar{\Phi} \simeq H_1(B, \zz) \otimes \cc$,
where $H_1(B, \zz)$ is the singular homology of $B$
with integral coefficients.

\begin{defi}
A \emph{CM datum} is a quintuple
$(F, \Phi, \langle \cdot, \cdot \rangle, \alpha, \Lambda)$ 
where $(F, \Phi)$ is a CM type with $\Phi$ of dimension $g$ over $\cc$,
$\Lambda$ is a $2g$-dimensional $\hat{\zz}$-lattice,
$\langle \cdot, \cdot \rangle$ is a skew-symmetric integral form 
on $\Lambda$,
and $\alpha \colon \oo_F \to \mend(\Lambda)$ is an action
of $\oo_F$ on $\Lambda$, satisfying formula \eqref{cmform}.
We additionally require that $\Lambda$ is compatible with $\Phi$: i.e., for
some (not necessarily unique) $\zz$-lattice $\Lambda_0$,
we have $\Lambda_0 \otimes \hat{\zz} \simeq \Lambda$
and $\Lambda_0 \otimes \cc \simeq \Phi \oplus \bar{\Phi}$
as symplectic $\oo_F$-modules.
\end{defi}

We say that a polarized
abelian variety $B$ has CM
datum
$(F, \Phi, \langle \cdot, \cdot \rangle, \alpha, \Lambda)$
if there is an isomorphism
$H_1(B,\hat{\zz})\cong \Lambda$ with $\oo_F$-action and polarization form
induced by $\alpha$ and $\langle \cdot, \cdot \rangle$ respectively. 
We will often denote the CM datum by a single letter,
$D=(F, \Phi, \langle \cdot, \cdot \rangle, \alpha, \Lambda)$.

More generally, suppose $X$ is a reduced scheme over $K$ and
that $B \to X$ is a strong $(F, \Phi)$-family 
of abelian varieties.
We say this family has CM datum $(F, \Phi, \langle \cdot, \cdot \rangle, \alpha, \Lambda)$ 
if the $\oo_F$-action induces
CM datum $(F, \Phi, \langle \cdot, \cdot \rangle, \alpha, \Lambda)$
on each $\cc$-fiber.
(A definition over arbitrary base schemes is possible
but more involved; this one suffices for our purposes.)

One can ask whether the functor of families of $(F, \Phi)$-abelian varieties
with CM datum $D = (F, \Phi, \langle \cdot, \cdot \rangle, \alpha, \Lambda)$ is
representable. In general, it is not representable
as a scheme, but is so as a (Deligne-Mumford) stack.
We will consider here for simplicity the \emph{coarse}
moduli space $S_{D}$
of such abelian varieties --- the initial scheme
which admits a map from this moduli stack
(or via the Yoneda embedding, the initial scheme
$S$ such that $\hom(-,S)$ admits a map from
the functor $\text{CM}_{D}$ which assigns to a 
scheme the set of families over it with CM by
${D}$). Points of the
coarse moduli space correspond canonically to $(F, \Phi)$-abelian
varieties with CM datum $D$ up to isomorphism.

Let $(F, \Phi)$ be a CM type defined over a field $K$,
and from now on suppose that the moduli space of strong $(F, \Phi)$-abelian varieties
is nonempty. Equivalently (by Lemma~\ref{nonempty}), suppose
that $\Phi \oplus \bar{\Phi} \simeq F^b \otimes_\qq K$ for some
integer $b$.
Note that the moduli space of strong $(F, \Phi)$-abelian varieties
is a disjoint union over all possible sets of CM data $D$
of the moduli space of strong $(F, \Phi)$-abelian varieties
with CM datum $D$.
In order to prove Theorem~\ref{irredthm}
we will identify this latter moduli space with a Shimura variety
(defined below).

\begin{defi} If $G$ is an algebraic group over $k$,
and $K$ is an extension of $k$, we define
\[G_{/K} = G \times_{\spec k} \spec K.\]
\end{defi}

\begin{defi} We write $\st = \res^\cc_\rr \gm{\cc}$.
\end{defi}

\begin{defi} Let $G$ be an algebraic group over the rationals, 
$h \colon \st \to G_{/\rr}$ be a map of real algebraic groups,
and $G_\oo \subset G(\af^f)$ be a compact subgroup.
Writing $K_\infty$ for the stabilizer of $h$ in $G(\rr)$,
we define the \emph{Shimura variety} to be the double coset space
\[\sh_{G_\oo}(G, h) := K_\infty \times G_\oo \backslash G(\af) / G(\qq).\]
\end{defi}

We will be interested in a particular class of Shimura varieties,
corresponding to particular triples $(G, h, G_\oo)$ which we now describe.
Let $D = (F, \Phi, \langle \cdot, \cdot \rangle, \alpha, \Lambda)$
be a CM datum, and $\Lambda_0$ be a $\zz$-lattice with
$\Lambda_0 \otimes \hat{\zz} \cong \Lambda$
and $\Lambda_0 \otimes \cc \simeq \Phi \oplus \bar{\Phi}$
as symplectic $\oo_F$-modules.
While $\Lambda_0$ is noncanonical, $V = \Lambda_0 \otimes \qq$
(which we think of as a vector space over $F$)
is uniquely determined by $D$ by the Hasse-Minkowski principle for quadratic forms.

\begin{defi} Define $\gsp(V)$ to be the group of $F$-linear maps
$V \to V$ preserving $\langle \cdot, \cdot \rangle$ up to a (rational) scalar.
We think of this as an algebraic group over $\qq$.
\end{defi}

Take $G = \gsp(V)$, and let $G_\oo \simeq G(\oo_{\af^f}) \subset G(\af^f)$
be the stabilizer of $\Lambda \otimes_\zz \oo_{\af^f}$.
Let $h \colon \st \to G_\rr$ be the unique map
of real algebraic groups
inducing multiplication by $z^{-1}$ on $\Phi$ and by $\bar{z}^{-1}$
on $\bar{\Phi}$ for an isomorphism
$V \otimes_\qq \cc \simeq \Phi \oplus \bar{\Phi}$
as $F$-$\cc$ bimodules.  

\begin{cthm}[Deligne, \cite{travshim}, paragraph~4.12]
The coarse moduli space of strong $(F, \Phi)$-abelian 
varieties with given CM datum $D$
is the above
Shimura variety.
\end{cthm}

Now, we recall some theorems of Shimura which
describe when the
irreducible components of $\sh_{G_\oo}(G, h)$
are geometrically irreducible.
To state these theorems, first observe that
$\operatorname{Hom}(\st_{/\cc}, \mathbb{G}_{m, \cc})$
has for a basis the characters $z$ and $\bar{z}$
such that the composition
$\cc^\times \simeq \st(\rr) \hookrightarrow \st(\cc) \to \mathbb{G}_{m, \cc} \simeq \mathbb{C}^\times$
is the identity and complex conjugation respectively.
Write $r \colon \mathbb{G}_{m, \cc} \to \st_{/\cc}$
for the unique map such that $z \circ r$ is the identity
and $\bar{z} \circ r$ is trivial.

\begin{cthm}[Shimura; see Deligne \cite{travshim}, paragraphs 3.6, 3.7, 3.9 and 3.14]
Suppose that the commutator subgroup $G' = [G, G]$ of $G$
is simply connected.
Then $\sh_{G_\oo}(G, h)$ is defined over $K$ if and only
if the conjugacy class 
of the composite map
\[h \circ r \colon \mathbb{G}_{m, \cc} \to \st_\cc \to G_\cc\]
is defined over $K$.
Moreover, writing $T = G^\text{ab} = G / G'$
and $\nu \colon G \to T$ for the projection map,
the irreducible components of $\sh_{G_\oo}(G, h)$
are geometrically irreducible if and only if
\[\lambda((\res^K_\qq \mathbb{G}_{m, K})(\af)) \subset \nu(G_\oo) \cdot T(\qq) \cdot \nu(K_\infty),\]
where the map $\lambda \colon \res^K_\qq \mathbb{G}_{m, K} \to T$ is defined by
\[\lambda(a) = \nm^K_\qq \circ \res^K_\qq (\nu \circ h \circ r)(a^{-1}).\]
\end{cthm}

\noindent
We now use these results to prove Theorem~\ref{irredthm}.

\begin{proof}[Proof of Theorem~\ref{irredthm}]
By the above discussion, it suffices to show
that the Shimura varieties $\operatorname{Sh}_{G_\oo}(G, h)$
are defined over $K$ provided that $\Phi$ is 
defined over $K$, and that in this case, every irreducible component
of $\operatorname{Sh}_{G_\oo}(G, h)$ is geometrically irreducible
if and only if $f = \det_F \Phi$ induces the zero map
$\cla(K) \to \cla^W(F)$.
Note that $G'$ is simply connected, since it is isomorphic
to a special unitary group over $F$. Thus, we can invoke
the theorems of Shimura above.

That $\operatorname{Sh}_{G_\oo}(G, h)$ is defined over $K$
is clear: Because
$\Phi$ is defined over $K$, the conjugacy class of
$r \circ h$ is therefore defined over $K$.

To see that every irreducible component
of $\operatorname{Sh}_{G_\oo}(G, h)$ is geometrically irreducible
if and only if $f \colon \cla(K) \to \cla^W(F)$ is zero,
we first notice that
\[T = \left\{(x, y) \in \left(\res^F_\qq \mathbb{G}_{m, F}\right) \times \mathbb{G}_{m, \qq} : x \cdot \bar{x} = y^b\right\}.\]
Here, the abelianization map
$\nu \colon G \to T$ is given explicitly by $g \mapsto (\det g, a)$,
where $a$ is the unique element of $\qq$ for which $\langle gx, gy \rangle = a \cdot \langle x, y \rangle $.
Under this identification, the map $\lambda$ is
given by
\[\lambda(a) = \left(\det_F\big((x \mapsto a \cdot x) \colon \Phi \to \Phi\big), \nm^K_\qq a \right).\]

Thus, the irreducible components of $\operatorname{Sh}_{G_\oo}(G, h)$ 
are geometrically irreducible if and only if for all $a \in K^\times(\af)$,
we have $\lambda(a) \subset \nu(G_\oo) \cdot T(\qq) \cdot \nu(K_\infty)$.

As $\nu(G_\oo) = T(\oo_{\af^f})$ and $\nu(K^\infty)$ contains
the connected component of $T(\rr)$, the above condition
is equivalent to the assertion that for all $a \in K^\times(\af)$,
we can find
$x \in F^\times$ so that $(f(a)) = (x)$ and
$x \cdot \bar{x} \in (\qq^\times)^b$.
Since $(f(a) \cdot \bar{f(a)}) = (g(a))^b$
and $x \cdot \bar{x} > 0$, this is equivalent
to $x \cdot \bar{x} \in \qq^\times$.
But this is just the assertion that
$f \colon \cla(K) \to \cla^W(F)$ is the zero map.
\end{proof}

This concludes the proof of Theorem~\ref{irredthm}.
As a corollary of Theorem~\ref{irredthm}, we can reprove
a well-known fact about the field of definition of abelian varieties
with full CM (see for example \cite{milne}).
\begin{cor} \label{full}
If $\dim_K \Phi = \frac{[F:\qq]}{2}$, i.e.\ the CM is \emph{full},
then there exists an $(F, \Phi)$-abelian variety
$B$ defined over $K_{F, \Phi}$.
\end{cor}
\begin{proof}
In this case the Shimura variety $\operatorname{Sh}_{G_\oo}(G, h)$
is zero-dimensional, hence any geometrically 
irreducible component over $K_{F, \Phi}$ corresponds
to a single (strong)
$(F, \Phi)$-abelian variety whose field of moduli is $K_{F, \Phi}$.
Moreover, by \cite{defmoduli}, this abelian variety can be defined over
its field of moduli.
\end{proof}

\noindent
Finally, using a theorem of Rizov, we have:

\begin{cor}
The intersection of all fields $K$ over which one can define
an $(F, \Phi)$-abelian variety is $K_{F,\Phi}$. 
\end{cor}
\begin{proof}
This follows from Theorem~\ref{irredthm} by a theorem of Rizov \cite{rizov}.
\end{proof}

\section{\label{sec:local} \boldmath Associated Characters of $A[\ell]$: Local Properties}

In this section, we give some local properties of the
associated characters of $A[\ell]$, most importantly Lemmas~\ref{obvs}
and~\ref{localnorm}, and Corollary~\ref{psiunr}.

\subsection{\label{sec:frobenii} Action of Frobenius Elements}
Fix a prime $v\nmid \ell$.
If $A$ has good reduction at $v$,
then $\rho_{A, \ell^\infty}$
is unramified at $v$ and the Frobenius element at $v$ acts by a matrix whose characteristic
polynomial $P_v$ is defined over $\zz$, and which has the additional property that all roots
of $P_v$ have (complex) norm $\sqrt{|\nm(v)|}$. (Note that this implies
a bound on the coefficients of $P_v$ given $v$.) When instead of $\rho$ we 
consider an associated character $\psi$, this implies that
$\psi$ at a prime of good reduction is unramified, and the image of a 
Frobenius element satisfies a polynomial all of whose roots
are a product of $d$ roots of $P_v$.

When we drop the assumption that $v$ has good reduction, 
the character $\psi$ can become ramified at $v$, but
a very similar result still holds --- in particular,
the image under $\psi$ of any Frobenius element $\pi$ of $v$ satisfies
one of a finite set (depending on $v$) of polynomials with integral coefficients.
Namely, we have the following result of Grothendieck.

\begin{lm}[Grothendieck]\label{frobpoly}
Let $v\in\Sigma_K\setminus \Sigma_\ell$ be any prime not dividing $\ell$.
Then for any choice of Frobenius
element $\pi = \pi_v \in G_K$ extending $\text{frob}_v$,
the characteristic polynomial $P_\pi$
of $\pi$ acting on the $\ell$-adic Tate module $A[\ell^\infty]$ has
the following properties.
\begin{enumerate}
\item \label{prop:int}
The coefficients of $P_\pi$ are integers, and are independent of $\ell$.
\item \label{prop:mag}
Every root of $P_\pi$ has magnitude
that does not depend on the choice of complex embedding, and is
equal to $1$, the norm of $v$, or
the square root of the norm of $v$.
\item \label{prop:pair}
The roots of $P_\pi$ come in pairs which multiply to the norm of $v$.
\end{enumerate}
\end{lm}

\begin{proof}
Properties \ref{prop:int}
and \ref{prop:mag} are Theorem~4.3(b) and Corollary~4.4
in \cite{monodromy} respectively;
property~\ref{prop:pair} follows from the
Galois-invariance of the Weil pairing.
\end{proof}

Note that $\psi(\pi)$ is the product of some subset of eigenvalues of 
$\rho_{A,\ell}(\pi)$. 
In light of the above lemma, this means that $\psi(\pi)$ is 
the reduction modulo some prime ideal $\pell\subset\oo_{\qqbar}$ 
lying over $\ell$ of
a product of $d$ roots of $P_v$ in $\qqbar$.
\begin{defi}
For any prime $v$ of $K$, choose a Frobenius element $\pi$ over $v$
and define 
\[\psi_\cc(v)\in\qqbar\]
to be a product of $d$ distinct roots of $P_\pi$ such that
$\psi(\pi)\equiv\psi_\cc(v)$ mod $\pell$
(There may be several choices for $\psi_\cc(v)$, but at least
one exists by the above lemma. We make such a choice for
each $v$.)
\end{defi}
\begin{lm}\label{obvs}
For any prime ideal $v \nmid \ell$, there are only finitely many
possible values of $\psi_\cc(v)$, all of which are algebraic
of degree at most $\binom{2g}{d}$. Moreover, there is
some integer $a$ with $0 \leq a \leq 2d$
such that under any complex embedding,
\[|\psi_\cc(v)| = \sqrt{|\nm^K_\qq(v)|}^{a}.\]
\end{lm}
\begin{proof}
This is a direct consequence of Lemma~\ref{frobpoly}.
(To see that $\psi_\cc(v)$ is algebraic
of degree at most $\binom{2g}{d}$, note that
any conjugate of $\psi_\cc(v)$ in $\qqbar$ is a product of $d$ roots
of the Frobenius polynomial $P_\pi\in \zz[x]$,
and in particular there are at most
$\binom{2g}{d}$ of them.)
\end{proof}

\subsection{\label{sec:semistable} Semistable Reduction}

It is well known that any abelian variety $A$ becomes
semistable after a finite extension of the ground field. In order to analyze
the local action of $\rho_{A,\ell}$ in Section \ref{sec:inertia},
we will need to make some more precise statements, which let us
control the ramification in a careful way.

\begin{defi}\label{defcg}
We define the constant $c_g$ to be the least
common multiple of all integers $n$ such that for all
but (at most) one odd prime $p$,
the symplectic group
$\spg_{2g}(\ff_p)$ has an element of order $n$.
\end{defi}

\begin{rem} For an alternative definition of $c_g$,
see Theorem~\ref{thmcg}.
\end{rem}

\begin{lm} \label{inertiabound}
Let $v$ be a valuation on $K$.
Then there exists a (finite and Galois) extension $L$ of $K$
such that $A$ acquires semistable reduction at $v$ over $L$
and the exponent of the inertia subgroup at $v$ of $\gal(L / K)$ divides $c_g$.
\end{lm}

\begin{proof}
For any odd prime $p$ distinct from the residue characteristic of $v$,
write $L^{(p)}$ for the field obtained
by adjoining the $p$-torsion points of $A$ to $K$.
By Proposition~4.7 of \cite{monodromy},
$A$ acquires semistable reduction over each $L^{(p)}$.

For any set $S$ of prime numbers, define $c_S(g)$
to be the least common multiple of all integers
$n$ such that for all odd primes $p \in S$, the symplectic group
$\spg_{2g}(\ff_p)$ has an element of order $n$.
Clearly, we can select some finite set
$S = \{p_1, p_2, \ldots, p_k\}$ of primes, all distinct
from the residue characteristic of $v$, so that $c_S(g) = c_g$.
Moreover, we can select this set of primes such
that none of them divide the degree of some fixed polarization of $A$.

Write $F = L^{(p_1)} \cdot L^{(p_2)} \cdots L^{(p_k)}$
for the composite field, and $I \subset \gal(F / K)$
for the inertia subgroup at $v$. Define $L \subset F$ to be
the fixed field of the subgroup generated by the kernels of the restriction
maps on $I$, i.e.\
\[L = \text{fixed field of} \ \left(\left\langle \ker\left(I \to \gal\big(L^{(p_i)} / K\big)\right) \colon i = 1, 2, \ldots k\right\rangle \subset \gal(F/K)\right).\]

Then an element of $\gal(L/K)$ is the image of an element in $\gal(L^{p_i}/K)$ for any $i$, so
its order divides $c_g$.
To see that $A$ acquires good reduction over $L$,
fix some prime $q$ distinct from the residue
characteristic of $v$ and all of the $p_i$.

As proved in subsection~4.1 of \cite{monodromy},
the subset $I' \subset I$ of elements
which act unipotently on the Tate module $T_q A$ forms a normal subgroup.
Proposition~3.5 of \cite{monodromy} implies that
for all $p_i$, the kernel
$\ker\big(I \to \gal(L^{(p_i)} / K)\big)$ acts unipotently
on $T_q A$; hence the kernel belongs to $I'$.
Thus, the subgroup spanned by the kernels belongs to $I'$,
and therefore acts unipotently
on $T_q A$. Applying Proposition~3.5 of \cite{monodromy}
again, we see that $A$ acquires semistable reduction at $v$ over $L$.
\end{proof}

\begin{cor} \label{psiunr}
The character $\psi^{c_g}$ is unramified at all places $v$ not lying over
$\ell$.
\end{cor}
\begin{proof}
By Proposition 3.5 of \cite{monodromy},
$\rho_{A,\ell,L}$ is unramified at $v$ if $v$ is
semistable for $A$ over $L$.
Now $g^{c_g}\in G_L\subset G_K$ for any $g\in G_K$, hence
the corollary.
\end{proof}

\begin{lm} \label{totram}
Let $v \mid \ell$, and suppose $\ell \nmid c_g$.
Then, for any abelian variety $A$,
there exists an extension $L_w / K_v$
over which $A$ acquires semistable reduction
and such that $[L_w : K_v] \mid c_g$.
\end{lm}
\begin{proof}
Write $L^{0}$ for the field $L$
given by Lemma~\ref{inertiabound},
and let $w_0$ be an extension of $v$ to $L^{0}$.
Since $\ell \nmid c_g$,
it follows that $L^0_{w_0} / K_v$ is tamely ramified.
In particular, its inertia subgroup is cyclic.
Thus, the order $e$ of its inertia subgroup divides $c_g$.

Write $I_v \subset G_v = \gal(\bar{K_v}/K_v)$
for the inertia subgroup and
$G_\text{res} = \gal(\bar{k_v} / k_v)$
for the Galois group of the residue field.
Note that there are noncanonical splittings
$G_\text{res}\to G_v$ making $G_v$ into a semidirect product of $I_v$ and $G_\text{res}$. Fix
such a splitting. Now if $H_0$ is the normal subgroup corresponding to $L^{0}_{w_0}$ then $H_0\cap I_v$ is normal
in $G_v$ and we can take for $L_w$ the field corresponding to the (not necessarily normal)
subgroup $H_0\cap I_v\times G_\text{res}$.
\end{proof}

\begin{rem}
Since in what follows we will only be interested in the action of $I_v$,
one could also work with the extension 
$K_v^{nr}\cdot L^{0}_w/K_v^{nr}$ 
instead of having to choose a splitting and working with $L_w/K_v$. 
We will use the essentially equivalent pair $L_w/K_v$ to avoid dealing
with infinite $p$-adic extensions.
\end{rem}

\subsection{\label{sec:inertia} Action of Inertia Groups}

Suppose at first that $v \mid \ell$ is semistable for $A$. Then a remarkable
result of Raynaud \cite{raynaud}
shows that the associated characters of $\rho_\ell$
restricted to $I_v$ are
algebraic over $\ell$ in the sense of
section~\ref{sec:algchars}.

\begin{lm}\label{localnormss}
If $A$ has semistable reduction at $v \mid \ell$, then $\psi|_{I_v}$
is the reduction of a
negative algebraic character of degree at most $d$ times the
ramification index of $v$.
\end{lm}

\begin{proof}
As explained in \cite{monodromy}, subsection~2.2.3,
we have a canonically defined subrepresentation
$A[\ell]^\text{f} \subset A[\ell]$
which comes from a finite flat
commutative group scheme over $\spec \oo_{K_v}$.
Thus, when restricted to the inertia
subgroup, the associated character $\psi$
decomposes as a product of at most $d$ associated characters
of $A[\ell]^\text{f}$ and of the quotient $A[\ell] / A[\ell]^\text{f}$.
For the associated characters of $A[\ell]^\text{f}$,
we are done by Corollary~3.4.4 of \cite{raynaud}.
On the other hand, Proposition~5.6 of \cite{monodromy}
implies that the action of the Galois group on
$A[\ell] / A[\ell]^\text{f}$ is unramified at $v$,
i.e.,\ each associated character of the restriction
to the inertia subgroup is trivial.
\end{proof}

\begin{lm}\label{localnorm}
Let $v \mid \ell$, and suppose $\ell \nmid \Delta_K \cdot c_g$.
Then, for any
abelian variety $A$,
the restriction $\psi^{c_g}|_{I_v}$ is the reduction
of a negative algebraic
character of degree at most $d \cdot c_g$.
\end{lm}
\begin{proof}
Let $L_w$ be the field given by Lemma~\ref{totram};
write $e = [L_w : K_v]$.
Functoriality of class field theory
tells us that the map 
$\rho_{A,\ell} \colon \oo_{L_w}^\times\to \bar{\ff}_\ell^\times$
giving the action of $\gal(L^{\text{ab}}/L)$ on
$A[\ell]$ is defined by composition with the norm map
$\nm^{L_w}_{K_v} \colon \oo_{L_w}^\times \to \oo_{K_v}^\times$,
as follows
\[\begin{diagram}
\oo_{K_v}^\times & \rTo                  & \oo_{L_w}^\times           &            &                      \\
                 & \rdTo_{x \mapsto x^e} & \dTo~{{\nm^{L_w}_{K_v}}}   & \rdTo^\psi &                      \\
                 &                       & \oo_{K_v}^\times           & \rTo_\psi  & \bar{\ff}_\ell^\times
\end{diagram}\]

Since $\ell \nmid \Delta_K$,
the ramification index of $w$ is at most $e$.
In addition, for any $u \in \oo_{K_v}^\times$, we have
\[\psi(u)^{c_g} = \psi(u^e)^{c_g/e} = \psi\left(\nm^{L_w}_{K_v} u\right)^{c_g / e}\]
which is the reduction of
a negative algebraic character of degree at most $d \cdot c_g$
by Lemma~\ref{localnormss}.
\end{proof}

\section{\boldmath Associated Characters of $A[\ell]$: Global Analysis \label{sec:global}}

In this section, we will patch the local information
from the previous section together into global information
in order to deduce the main theorem.

\begin{rem} \label{general}
In fact, we will prove the main theorem
using only Lemmas~\ref{obvs}
and~\ref{localnorm}, and Corollary~\ref{psiunr}
from the previous section
(as well as the fact that $c_g$ is even).
In particular, this means that the conclusion
of the main theorem holds more
generally for any Galois character satisfying
these three results.
(When we make the main theorem effective in Section~\ref{sec:effective},
we will for concreteness use the more explicit Lemma~\ref{frobpoly}
as well.)
\end{rem}

\subsection{\boldmath The Character $\theta^S$}

For the remainder of the paper, we can assume $\ell \nmid c_g \cdot \Delta_K$.
Together, the data of $\psi$ at all inertia groups lets us 
reconstruct $\psi^{c_g}$
on the subgroup of $G_K$ fixing the Hilbert class field. 
Namely
let $U\subset\ii$ be the group of units.

\begin{lm}\label{serrelocal}
There is a positive algebraic character $\theta^F$
of degree at most $d \cdot c_g$ such that
the restriction $\psi^{c_g}|_U \equiv (\theta^F_\ell)^{-1}$
mod $\mathfrak{l}$.
\end{lm}
\begin{proof}
This is an immediate consequence of 
Lemma \ref{localnorm} and the fact that $\psi^{c_g}$
is unramified at $v\nmid \ell$ (by Lemma \ref{psiunr}).
(Note that since $c_g$ is even, we don't
need to worry about the infinite places.)
\end{proof}

\begin{defi} \label{corresp}
We say that a pair $(S, e)$, where $e \mid c_g$ and
$S \in \zz[\emb{K}]$ is of degree $d \cdot e$
\emph{corresponds} to an associated character $\psi$ if
for all $x \in K^\times$ be relatively prime to $\ell$,
\[\psi(x_{\hat{\ell}})^{c_g} \equiv \theta^S(x)^{c_g/e} \mod \mathfrak{l}.\]
If $e$ is coprime to $S$ as an element of $\zz[\emb{K}]$,
we say that $(S, e)$ is \emph{reduced}.
\end{defi}

\begin{lm} \label{lmga}
Every associated character corresponds to a reduced pair $(S, e)$.
(When $A$ is semistable at all primes lying over $\ell$, we can take $e = 1$.)
\end{lm}
\begin{proof}
Let $F \in \zz[\emb{K}]$
be the index from Lemma~\ref{serrelocal}.
Define $f$ to be the greatest
common divisor of $c_g$ and the $F$,
and write $e = c_g / f$.
By Lemma~\ref{localnormss}, $e = 1$ when
$A$ is semistable at all primes over $\ell$.
We define 
\[S = \frac{F}{f} \in \zz[\emb{K}].\]
Since $\psi(x_{\hat{\ell}}) \cdot \psi(x_\ell) = 1$ for $x \in K^\times$,
and $\theta^F(x) = (\theta^S(x))^{c_g / e}$, we are done by Lemma~\ref{serrelocal}.
\end{proof}

\subsection{\label{sec:charclass} \boldmath Analysis of the Character $\theta^S$}

\begin{defi}
We adopt the notation
``$\ell$ sufficiently large'' to mean ``$\ell$ larger than
a constant depending only
on $K$ and $g$.''
\end{defi}

For the rest of this section, we fix $K$
and one of the $(d \cdot c_g + 1)^{n_K}$ possible
reduced pairs $(S, e)$, and
we assume $(S, e)$ corresponds to an associated
character of an abelian variety.
Here we will give ineffective bounds;
we will make these arguments effective in section~\ref{sec:effective}.

\begin{lm} \label{thetaisbalanced}
For $\ell$ sufficiently large, the character $\theta^S$ is balanced.
\end{lm}
\begin{proof}
If the character $\theta^S$ is not balanced, then there is some unit $u$
for which $\theta^S(u)$ is not a root of unity by Lemma~\ref{restons}.
However, we have
\[\theta^S(u)^{c_g/e} \equiv \psi(u_{\hat{\ell}})^{c_g} \equiv 1 \mod \mathfrak{l}\]
which is a contradiction for $\ell$ sufficiently large.
\end{proof}

\begin{defi} We define $h_K'$ to be the exponent of the class group $\cla(K)$.
\end{defi}

\begin{lm} \label{equality}
Let $v$ be a prime ideal. As $v^{h_k'}$ is principal,
we can write $v^{h_K'} = (x)$. Then if $\ell$ is sufficiently large
relative to $v$, we have
\[\psi_\cc(v)^{c_g \cdot h_K'} = \theta^S(x)^{c_g/e}.\]
\end{lm}
\begin{proof}
If $\ell$ is sufficiently large relative to $v$,
then $v$ does not lie over $\ell$.
Thus, for any choice of Frobenius element $\pi_v$ at $v$,
Lemma~\ref{lmga} implies
\begin{equation}\label{classfieldeqn}
\psi\left(\pi_v^{h_K'}\right)^{c_g} \equiv \theta^S(x)^{c_g/e} \mod \mathfrak{l}.
\end{equation}
Hence $\ell$ divides the norm of their difference.
By Lemma~\ref{obvs},
there are only finitely many possibilities
for the left-hand side as $A$ ranges over all abelian varieties
of dimension $g$.
So if $\ell$ is sufficiently large, then we have the desired equality.
\end{proof}

\begin{lm} \label{lmcc} For $\ell$ sufficiently large,
there is a fixed integer $a$
with $0 \leq a \leq 2d$ such that if $\sigma$ and $\tau$
are complex conjugate embeddings (for any choice of embedding
$\qqbar \hookrightarrow \cc$), then
\[S(\sigma) + S(\tau) = ae.\]
\end{lm}

\begin{proof} In light of Lemma~\ref{thetaisbalanced},
there is an integer $a'$ such that
$S(\sigma) + S(\tau) = a'$ for any complex conjugate
embeddings $\sigma$ and $\tau$.
Hence, it suffices to show that this integer $a'$
satisfies $a' = a \cdot e$
for $a$ an integer with $0 \leq a \leq 2d$.

Let $v$ be a prime ideal of $K$, and write $v^{h_K'} = (x)$.
From Lemma~\ref{obvs},
there exists some integer $a$ with $0 \leq a \leq 2d$
such that under any complex embedding,
\[|\psi_\cc(v)| = \sqrt{|\nm^K_\qq(v)|}^{a}.\]
On the other hand, under any complex embedding,
we have
\[|\theta^S(x)| = \sqrt{|\nm^K_\qq(x)|}^{a'} = \sqrt{|\nm^K_\qq(v)|}^{a' \cdot h_K'}.\]
By Lemma~\ref{equality}, we have
\[\psi_\cc(v)^{c_g \cdot h_K'} = \theta^S(x)^{c_g/e} \quad \Rightarrow \quad \psi_\cc(v)^{e \cdot c_g \cdot h_K'} = \theta^S(x)^{c_g}\]
Combining these, we have
\[\sqrt{|\nm^K_\qq(v)|}^{a \cdot e \cdot c_g \cdot h_K'} = |\psi_\cc(v)|^{e \cdot c_g \cdot h_K'} = |\theta^S(x)|^{c_g} = \sqrt{|\nm^K_\qq(v)|}^{a' \cdot c_g \cdot h_K'}.\]
Hence $a' = a \cdot e$, which is what we wanted to show.
\end{proof}

\begin{defi} We define $F \subset \qqbar$ to be the smallest field
containing the image of $\theta^S$.
\end{defi}

\begin{lm} \label{notpower}
Let $v \subset \oo_K$ be degree~$1$ and unramified in $K / \qq$.
Then there is no factor $e'$ of $e$ such that $\theta^S(v)$ is an
$(e')$th power in the group of ideals of $F$.
\end{lm}
\begin{proof}
By assumption, the set of exponents to which primes occur in the prime
factorization of $\theta^S(v)$ are the same as coefficients of $S$.
In particular, $e$ is coprime to the greatest common divisor
of all these coefficients.
\end{proof}

\begin{lm} \label{degoneprime}
Let $\ell$ be sufficiently large;
suppose that $v \subset \oo_K$ is a prime ideal of degree~$1$,
unramified in $K/\qq$. Write $v^{h_K'} = (x)$.
Then $\theta^S(x)^{c_g/e}$ generates $F$ over $\qq$.
\end{lm}

\begin{proof}
Take any $\tau \in \gal(\kgal/\qq)$
which fixes $\theta^S(x)^{c_g/e}$.
We want to show that $\tau$ fixes the field $F$.
Since $\tau$ fixes $\theta^S(x)^{c_g/e}$, we have an equality of ideals
of $\kgal$:
\[\left(\prod_{\sigma \in \emb{K}} \sigma(v)^{S(\sigma)}\right)^{h_K' \cdot c_g / e} = \left(\prod_{\sigma \in \emb{K}} \tau\sigma(v)^{S(\sigma)}\right)^{h_K' \cdot c_g / e}.\]
Since the group of ideals of $\kgal$ is torsion-free, this implies
\[\prod_{\sigma \in \emb{K}} \sigma(v)^{S(\sigma)} = \prod_{\sigma \in \emb{K}} \tau\sigma(v)^{S(\sigma)} = \prod_{\sigma \in \emb{K}} \sigma(v)^{S(\tau^{-1} \sigma)}.\]
Because $v$ is an unramified prime of degree~$1$,
its images under distinct embeddings into $\kgal$ generate
coprime ideals of $\kgal$.
Hence, by the uniqueness of prime factorization into ideals for $\kgal$,
we have $S(\sigma) = S(\tau^{-1} \sigma)$ for all $\sigma \in \emb{K}$.
Thus, for any $z \in K^\times$, we have
\[\theta^S(z) = \prod_{\sigma \in \emb{K}} \sigma(z)^{S(\sigma)} = \prod_{\sigma \in \emb{K}} \sigma(z)^{S(\tau^{-1} \sigma)} = \prod_{\sigma \in \emb{K}} \tau\sigma(z)^{S(\sigma)} = \tau\theta^S(z),\]
so $\tau$ fixes the image of $\theta^S$ and hence the field $F$.
\end{proof}

In particular, the above lemma together with Lemma \ref{obvs}
implies that for $\ell$
sufficiently large, $F$ has degree at most $\binom{2g}{d}$.
In fact, we can prove a stronger statement.

\begin{lm}\label{lmcl}
If $\ell$ is sufficiently large, then for any ideal class $\mathfrak{v} \in \cla(K)$,
\[\ord_{\cla^W(F)}\big(C\theta^S(\mathfrak{v})\big) \cdot [F : \qq] \cdot e \leq \binom{2g}{d}.\]
\end{lm}

\begin{proof} By the Chebotarev Density Theorem,
we can find some prime ideal $v \subset \oo_K$
representing the ideal class $\mathfrak{v}$ which is
of degree~$1$ and unramified in $K/\qq$
(since the set of prime ideals of degree greater than $1$
or ramified in $K/\qq$ has density zero).
Write $v^{h_K'} = (x)$. By Lemma~\ref{equality},
\begin{equation}\label{foo}
\psi_\cc(v)^{c_g \cdot h_K'} = \theta^S(x)^{c_g / e}.
\end{equation}

By Lemma~\ref{obvs}, $\psi_\cc(v)$
lies in some field $L$ of degree at most $\binom{2g}{d}$.
However, by Lemma~\ref{degoneprime}, the right-hand side
generates the field $F$ over $\qq$.
Thus, we have $F \subset L$.
Hence, we have an equality of ideals of $L$:
\[(\psi_\cc(v))^{c_g \cdot h_K'} = \big(\theta^S(x)^{c_g / e}\big) = \theta^S(v)^{c_g \cdot h_K' / e}.\]
Because the group of fractional ideals is torsion-free,
$(\psi_\cc(v))^e = \theta^S(v)$.
Taking norm down to $F$, we have an equality of ideals of $F$
\[(\nm^L_F \psi_\cc(v))^e = (\theta^S(v))^{[L:F]}.\]
Since the left-hand side is an $e$th power
in the group of ideals, the right-hand side must be as well;
thus $e \mid [L : F]$ by Lemma~\ref{notpower}.
Because the group of ideals is torsion-free,
\[(\nm^L_F \psi_\cc(v)) = (\theta^S(v))^{\frac{[L:F]}{e}}.\]
The left-hand side is zero in $\cla^W(F)$,
so the right-hand side is too.
This gives
\[\ord_{\cla^W(F)} \big(\theta^S(v)\big) \leq \frac{[L : F]}{e} = \frac{[L : \qq]}{[F : \qq] \cdot e} \leq \frac{\binom{2g}{d}}{[F : \qq] \cdot e},\]
which implies the desired inequality.
\end{proof}

\begin{lm} \label{lmun}
There exists an integer $N$ divisible by $eN_0$
(recall that $N_0$ is the number of roots of unity in $F$)
and satisfying $\phi(N) \leq \binom{2g}{d}$
such that when restricted to $\gal(\bar{K} / K_{F, \Phi})$,
\[(\psi|_{\gal(\bar{K} / K_{F, \Phi(S)})})^{e \cdot w} = \psi_{F, \Phi(S)}^w \quad \text{where} \quad w = \frac{1}{e} \cdot \lcm(N, c_g).\]
\end{lm}

\begin{proof}
Define
\[\chi = (\psi|_{\gal(\bar{K} / K_{F, \Phi(S)})})^{e \cdot N_0} \otimes \psi_{F, \Phi(S)}^{N_0} \quad \text{and} \quad w_0 = \frac{c_g}{\gcd(eN_0, c_g)}.\]
From Lemma~\ref{psiunr}, it follows that $\chi^{w_0}$
is unramified at all places not lying over $\ell$.
By assumption, it is trivial at all id\`eles
of the form $x_{\hat{\ell}}$ for $x \in K^\times$ relatively prime to $\ell$.
It follows that $\chi^{w_0}$
is trivial on $\gal(\bar{K} / H_K)$
(by the id\`elic formulation of class field theory).
That is, $\chi^{w_0}$ gives a well-defined character
$\chi^{w_0} \colon \gal(H_K / K_{F, \Phi}) \to \bar{\ff}_\ell^\times$.

Let $\mathfrak{v}$ be an ideal class such that $\chi^{w_0}(\mathfrak{v})$
generates the image of $\chi^{w_0}$ in $\bar{\ff}_\ell^\times$.
Let $v$ be a prime ideal
representing $\mathfrak{v}$ which is
of degree~$1$, does not divide $h_K' \cdot c_g$,
and is unramified in $K/\qq$.
Note that by definition,
\[\chi(\pi_v) =  X^{N_0} \quad \text{for} \quad X = \frac{\psi_\cc(v)^e}{g},\]
where $g \in F^\times$ is any Weil number such that $\theta^S(v) = (g)$.

Lemma~\ref{equality} implies that
$X^{h_K' \cdot c_g} = 1$
 for $\ell$ sufficiently large.
Note that by Lemma~\ref{degoneprime}, $F \subset \qq[\pi_\cc(v)^e]$, so
$\qq[X] \subset \qq[\psi_\cc(v)^e] \cdot F \subset \qq[\psi_\cc(v)^e]$.

By Lemma~\ref{notpower}, there is no factor $e'$ of $e$ for which
$(g) = \theta^S(v)$ is an $(e')$th power in the group of ideals of $F$.
But Lemma~\ref{equality} implies that
$\psi_\cc(\pi_v)^e$ equals $g$ times an $(h_K' \cdot c_g)$th
root of unity, so $\qq[\psi_\cc(\pi_v)^e] / F$ is unramified
at all primes dividing $\theta^S(v)$.
Hence, there is no factor $e'$ of $e$ for which
$(\psi_\cc(v)^e) = (g)$
is an $(e')$th power in the group of ideals of
$\qq[\psi_\cc(v)^e]$.
This implies that $[\qq[\psi_\cc(v)] : \qq[\psi_\cc(v)^e]] = e$.
Hence, Lemma~\ref{obvs} gives
\[[\qq[X] : \qq] \leq [\qq[\psi_\cc(v)^e] : \qq] = \frac{1}{e} \cdot [\qq[\psi_\cc(v)] : \qq] \leq \frac{1}{e} \cdot \binom{2g}{d}.\]

Since $\phi(em) \leq e \phi(m)$,
we can take $N$ to be $e$ times the number of roots of unity
in $\qq[X]$.
By construction, $eN_0 \mid N$ and $X^{N/e} = 1$, so
\[(\psi|_{\gal(\bar{K} / K_{F, \Phi(S)})})^{e \cdot w} = \psi_{F, \Phi(S)}^w \quad \text{where} \quad w = \gcd(N/e, N_0 \cdot w_0) = \frac{1}{e} \cdot \lcm(N, c_g). \qedhere\]
\end{proof}

\begin{thm}\label{thm:charclass}
Let $K$ be a number field and $g$ and $d$ be positive integers.
Suppose $\ell$ is a prime number not belonging to some finite set $S_{K, g}$
depending 
only on $K$ and $g$. Suppose moreover that
$A$ is a $g$-dimensional abelian variety defined over $K$
and $\psi_0$ is an associated character of $\rho_{A,\ell}$ of {\weight}~$d$.
Then, for an effectively computable integer
$c_g<12^{4g^2}$, there is a positive algebraic character
$\theta^S$ and positive integer $e$ such that
the following conditions are satisfied.
\begin{enumerate}
\item \label{prop:cc} The character $\theta^S$ is balanced, of
total degree $a \cdot e$ for some $0 \leq a \leq 2d$.
\item \label{prop:cl}
The induced map $C\theta^S \colon \cla(K) \to \cla^W(F)$ is trivial,
i.e.\ $K \supset K_{F, \Phi(S)}$.
\item  \label{prop:un}
There exists an integer $N$ divisible by $eN_0$
and satisfying $\phi(N) \leq \binom{2g}{d}$
such that
\[\psi^{e \cdot w} = \psi_{F, \Phi(S)}^w \quad \text{where} \quad w = \frac{1}{e} \cdot \lcm(N, c_g).\]
\item \label{prop:in}
We have the inequality
$[F : \qq] \cdot e \leq \binom{2g}{d}$.
\end{enumerate}
\end{thm}
\begin{proof}
Write $e'$ for the exponent of the induced map
$C\theta^S \colon \cla(K) \to \cla^W(F)$.
Then, after replacing $(S, e)$ with $(S \cdot e', e \cdot e')$,
this follows from Lemmas~\ref{lmcc}, \ref{lmcl} and \ref{lmun}.
\end{proof}

\subsection{\label{sec:varproof} Proof of Main Theorem}

\begin{thm} \label{mainthm}
Let $K$ be a number field, and $g$ and $d$ be positive integers.
Then, there exists a
finite set $S_{K, g}$ of prime numbers depending only on $K$
and $g$, and a constant $0 < c_g < 12^{4g^2}$ depending only on $g$
such that, for a prime $\ell \notin S_{K, g}$,
and a $g$-dimensional abelian variety $A$
with a mod-$\ell$ associated character $\psi_0$
of {\weight} $d$,
we have
\[\psi^{e \cdot w} \equiv \psi_{F, \Phi}^w \pmod \ell,\]
where $\psi$ is either $\psi_0$ or $\cyc_\ell^d \otimes \psi_0^{-1}$
and $w = \frac{\lcm(N, c_g)}{\gcd(e, c_g)}$.
Here, $F$ is either $\qq$ or a CM-field,
and $\Phi \colon F \to \mend(K^m)$
is a primitive balanced representation
such that $K \supset K_{F, \Phi}$.
The quantities $a$, $e$, and $N$ are integers with
$e$ and $N$ positive, which
satisfy $m = \frac{1}{2} \cdot a \cdot e \cdot [F : \qq]$.
Moreover, $0 \leq a \leq d$, and both
$\phi(N)$ and $e \cdot [F : \qq]$ are
at most $\binom{2g}{d}$.
\end{thm}

\begin{proof}
Let $(S, e)$ corresponding to $\psi$ be as in Theorem~\ref{thm:charclass}.
Note that $(S', e)$ corresponds to $\cyc_\ell^d \otimes \psi^{-1}$,
where $S'$ is defined by $S'(\sigma) = de - S(\sigma)$.
One can easily check that $(S', e)$ satisfies
the conclusion of Theorem~\ref{thm:charclass}
for the character $\cyc_\ell^d \otimes \psi^{-1}$.
Thus, by replacing $\psi$ with $\cyc_\ell^d \otimes \psi^{-1}$
if necessary, we may suppose that $0 \leq a \leq d$.

Let $\Phi$ be the representation
of $F$ such that $\det_F \Phi = \theta^S$ (see Proposition~\ref{psifphi}).
By construction,
$\dim \Phi = \frac{1}{2} \cdot [F : \qq] \cdot a \cdot e$.
The conclusion of this theorem then follows from Theorem~\ref{thm:charclass}.
\end{proof}

\begin{thm} \label{ssoverell}
In Theorem~\ref{mainthm}, if we assume in addition that $A$
is semistable at primes lying over $\ell$, then we can take $e = 1$,
provided that we weaken the conclusion to be
that $\psi^w = \psi_{F, \Phi}^w$
on $\gal(\bar{K} / K \cdot K_{F, \Phi})$
and
\[[F : \qq] \cdot (\exponent K \cdot K_{F, \Phi} / K) \leq \binom{2g}{d}.\]
\end{thm}

\begin{proof}
We use the same argument as in the proof of Theorem~\ref{mainthm},
without replacing $(S, e)$ by $(e' \cdot S, e' \cdot e)$.
Instead, we replace $K$
with the unramified abelian extension $M$ of $K$ defined
by the kernel of $C\theta^S \colon \cla(K) \to \cla^W(F)$.
\end{proof}

\subsection{\label{subsec:cors} Some Corollaries of the Main Theorem}

\begin{cor} \label{cor:nocmfield}
Let $K$ be a number field that does not contain
any CM-fields (which is in particular true when $K$
has a real embedding), and $g$ and $d$ be positive integers.
There exists a
finite set $S_{K, g}$ of prime numbers depending only on $K$
and $g$, and a constant $0 < c_g < 12^{4g^2}$ depending only on $g$
such that, for a prime $\ell \notin S_{K, g}$,
and a $g$-dimensional abelian variety $A$
with a mod-$\ell$ associated character $\psi$
of {\weight} $d$,
\[\psi^{2w} = \cyc_\ell^{aw},\]
where
$a$ is an integer with $0 \leq a \leq 2d$, and
$w = \frac{\lcm(c_g, N)}{2}$ for some positive $N \leq \binom{2g}{d}$.
\end{cor}

\begin{proof} Since $K \supset K_{F, \Phi}$
and $K_{F, \Phi}$ contains a CM-field when $F$ is a CM-field,
$F$ cannot be a CM-field. Thus, $F = \qq$, which
gives the desired conclusion.
\end{proof}

\begin{cor} \label{cor:ssone}
Let $K$ be a number field, and $g$ and $d$ be positive integers.
Then, there exists a
finite set $S_{K, g}$ of prime numbers depending only on $K$
and $g$, and a constant $0 < c_g < 12^{4g^2}$ depending only on $g$
such that, for a prime $\ell \notin S_{K, g}$,
and a $g$-dimensional abelian variety $A$
with a mod-$\ell$ associated character $\psi$
of {\weight} $1$,
one of the following holds.
\begin{enumerate}
\item The character $\psi^{c_g}$ is trivial or equal to $\cyc_\ell^{c_g}$.
\item There exists an abelian unramified extension $M/K$,
a (full) CM abelian variety $A'$ defined over $M$,
such that $K$ contains the reflex field of the CM field of $A'$
(which in particular implies that $A'$ has CM defined over $M$),
and an $\ell$-adic associated character of \weight~$1$
of $A'$, whose mod-$\ell$ reduction $\psi'$ satisfies
\[(\psi|_{\gal(\bar{K}/M)})^{c_g} = (\psi')^{c_g} \quad \text{and} \quad 
(\dim A') \cdot (\exponent M/K) \leq g.\]
\end{enumerate}
\end{cor}
\begin{proof}
Here, we use Theorem~\ref{ssoverell}.
The two cases correspond to $F = \qq$ and $F$ a CM-field.
In the second case,
we have $m = \frac{1}{2} \cdot [F : \qq]$,
so the result follows from Corollary~\ref{full}.
\end{proof}

The next two Corollaries list the possible
\weight~1 associated characters for $g \in \{2, 3\}$.
(Note that the easiest way of computing $c_g$ is to use
Theorem~\ref{thmcg} of Section~\ref{sec:effective} below.)

\begin{cor}\label{absurf}
Let $K$ be a number field.
Then there exists a
finite set $S_{K, 2}$ of prime numbers depending only on $K$
such that, for a prime $\ell \notin S_{K, 2}$,
and an abelian surface $A$
with a mod-$\ell$ associated character $\psi$
of {\weight} $1$,
one of the following holds.
\begin{enumerate}
\item There exists a full CM abelian surface $A'$
over $K$ whose CM is defined over $K$, with an
$\ell$-adic \weight~$1$ associated character
whose mod-$\ell$ reduction $\psi'$ satisfies
\[\psi^{120} = (\psi')^{120}.\]
\item
There exists an abelian unramified extension $L/K$ of exponent
at most $2$, a CM
elliptic curve $E'$ defined over $L$,
such that $K$ contains the CM field, and
an $\ell$-adic \weight~$1$ associated character of $E'$ whose mod-$\ell$
reduction $\psi'$ satisfies
\[\psi|_{\gal(\bar{K}/L)}^{120} = (\psi')^{120}.\]
\item For some $a \in \{0, 60, 120\}$, we have
\[\psi^{120} = \cyc_\ell^a.\]
\end{enumerate}
\end{cor}

\begin{cor}\label{abthree}
Let $K$ be a number field.
Then there exists a
finite set $S_{K, 3}$ of prime numbers depending only on $K$
such that, for a prime $\ell \notin S_{K, 3}$,
and an abelian threefold $A$
with a mod-$\ell$ associated character $\psi$
of {\weight} $1$,
one of the following holds.
\begin{enumerate}
\item
There exists a full CM abelian surface or threefold $A'$
over $K$ whose CM is defined over $K$, with an
$\ell$-adic \weight~$1$ associated character
whose mod-$\ell$ reduction $\psi'$ satisfies
\[\psi^{2520} = (\psi')^{2520}.\]
\item
There exists an abelian unramified extension $L/K$ of exponent
at most $3$, a CM
elliptic curve $E'$ defined over $L$,
such that $K$ contains the CM field, and
an $\ell$-adic \weight~$1$ associated character of $E'$ whose mod-$\ell$
reduction $\psi'$ satisfies
\[\psi|_{\gal(\bar{K}/L)}^{2520} = (\psi')^{2520}.\]
\item There exists a CM elliptic curve $E'$ over $K$,
such that $K$ contains the CM field, and
an $\ell$-adic \weight~$1$ associated character of $E'$ whose mod-$\ell$
reduction $\psi'$ satisfies
\[\psi^{2520} = (\psi' \otimes \cyc_\ell)^{840}.\]
\item For some $a \in \{0, 1260, 2520\}$, we have
\[\psi^{2520} = \cyc_\ell^a.\]
\end{enumerate}
\end{cor}

\section{\label{sec:curves} The Special Case of Elliptic Curves}

In this section, we assume in addition that $g = 1$,
i.e.\ that $A = E$ is an elliptic curve. In this
case, we can prove a slightly stronger theorem.

\subsection{\boldmath Theorem~\ref{mainthm} in the Case $g = 1$}

\begin{lm} \label{mainthmforcurves}
Let $K$ be a number field.
Then, there exists a
finite set $S_K$ of prime numbers depending only on $K$
such that, for a prime $\ell \notin S_K$,
and an elliptic curve $E$ over $K$ for which $E[\ell]\otimes\bar{\ff}_\ell$ 
is reducible with {\weight} $1$ associated character $\psi$,
one of the following holds.
\begin{enumerate}
\item \label{eccyc} We have $\psi^{12} \in \{1, \cyc_\ell^6, \cyc_\ell^{12}\}$.
\item \label{eccm} There exists a CM
elliptic curve $E'$, which is defined over $K$
and whose CM-field is contained in $K$, with
an $\ell$-adic {\weight} $1$ associated character
whose mod-$\ell$ reduction $\psi'$ satisfies:
\[\psi^{12} = (\psi')^{12}\]
\end{enumerate}
\end{lm}

\begin{proof}
If $F = \qq$ in Theorem~\ref{mainthm},
Case~\ref{eccyc} holds.
Otherwise, $F$ is imaginary quadratic, and
we can take $E' = \cc / \oo_F$,
since $\psi_{F, \Phi} = \psi'$ and $K_{F, \Phi}$ is the
Hilbert class field of $F$.
\end{proof}

\subsection{Effective Chebotarev Theorem \label{sec:cheb}}

Under GRH, we have the following effective version
of the Chebotarev Density Theorem, due to
Lagarias and Odlyzko, with improvements due to Bach.

\begin{cthm}[Effective Chebotarev]\label{thm:chebotarev}
Let $E/K$ be a Galois extension of number fields
with $E \neq \qq$.
Then under GRH, there exists an effectively computable absolute constant $\cceb$
such that every conjugacy class of $\gal(E/K)$ is represented
by a Frobenius element
of a prime ideal $v \in \Sigma_K$ which is unramified in $E$,
such that
\[\nm_{K/\qq}(v) \leq \cceb (\log \Delta_E)^2.\]
Moreover, we can take $v$ to be a prime of degree~1 and unramified in $K/\qq$.
\end{cthm}
\begin{proof}
See \cite{lo}, remark at end of paper
regarding the improvement to Corollary~1.2.
That we can take $\mathfrak{p}$ to be of degree~1
and unramified in $K/\qq$
follows from Theorem~3.1 
in~\cite{bach}.
\end{proof}

\begin{rem} \label{remuncond}
Unconditionally, a similar theorem is true,
with the bound of $\cceb (\log \Delta_E)^2$
replaced by $\exp(\cuncond\Delta_E)$, for an effectively
computable absolute constant $\cuncond$.
\end{rem}

In fact, we will need a very slight strengthening
of the above Theorem, to avoid possible issues
at the prime $3$.

\begin{cor}\label{cor:avo}
Let $E/K$ be a Galois extension of number fields
with $E \neq \qq$, and let $N$ be a positive (rational) integer.
Then under GRH, there exists an effectively computable
absolute constant $\ccebavo$
such that every conjugacy class of $\gal(E/K)$ is represented
by a Frobenius element
of a prime ideal $v \in \Sigma_K$ which is unramified in $E$,
such that
\[\nm_{K/\qq}(v) \leq \ccebavo \cdot (\log \Delta_E + n_E \log N)^2.\]
Moreover, we can take $v$ to be of degree~1, unramified in $K/\qq$,
and not divide $N$.
\end{cor}

\begin{proof}
This is proven in \cite{quelques}
for $K = \qq$. Thanks to Bach's improvement to the effective
Chebotarev theorem (that we can take $v$ to be unramified in $K/\qq$),
the same argument works for arbitrary number fields.
For completeness, we recall the proof below.

Clearly, we can assume that $N$ is squarefree.
Define $E' = E[\sqrt{N}]$.
Then every prime of $E'$ lying over a prime divisor of $N$
is ramified in $E'/\qq$.

Now, we apply effective Chebotarev to $\gal(E' / K)$,
to conclude that every conjugacy class of $\gal(E'/K)$
is represented by a Frobenius element of a
prime ideal $v \in \Sigma_K$ of degree~$1$ which is unramified
in $E'$ and in $K/\qq$, and thus is coprime to $N$,
such that $\nm^K_\qq(v) \leq \cceb (\log \Delta_{E'})^2$.
Since $E'/E$ is ramified only at primes dividing $2N$,
we can bound $\Delta_{E'}$ using Proposition~5 of
section~1.3 of \cite{quelques},
and thereby conclude we can take $v$ so that
\begin{align*}
\nm^K_\qq(v) &\leq \cceb (\log \Delta_{E'})^2 \\
&\leq \cceb \cdot \left(2 \log \Delta_E + n_E \log 2N + n_E \log 2\right)^2 \\
&\leq \ccebavo \cdot (\log \Delta_E + n_E \log N)^2
\end{align*}
for some effectively computable absolute constant $\ccebavo$.
\end{proof}

\subsection{Proof of Theorem~\ref{ellcurve}}

\begin{thm} \label{ellcurve}
Let $K$ be a number field.
Then, there exists a
finite set $S_K$ of prime numbers depending only on $K$
such that, for a prime $\ell \notin S_K$,
and an elliptic curve $E$ over $K$ for which $E[\ell]\otimes\bar{\ff}_\ell$ 
is reducible with {\weight} $1$ associated character $\psi$,
one of the following holds.
\begin{enumerate}
\item\label{ellcurve1} There exists a CM
elliptic curve $E'$, which is defined over $K$
and whose CM-field is contained in $K$, with
an $\ell$-adic {\weight} $1$ associated character
whose mod-$\ell$ reduction $\psi'$ satisfies:
\begin{equation*}
\psi^{12} = (\psi')^{12}
\end{equation*}
\item \label{ellcurve2}
The Generalized Riemann Hypothesis fails for $K[\sqrt{-\ell}]$, and 
\begin{equation*}
\psi^{12} = \cyc_{\ell}^6,
\end{equation*}
where $\cyc_\ell$ is the cyclotomic character.
(Moreover, in this case we must have $\ell \equiv 3$ mod $4$ and
the representation $\rho_{E, \ell}$ is already reducible over $\ff_\ell$.)
\end{enumerate}
\end{thm}

\begin{proof}
Since $\rho_{E, \ell}$ is reducible and $2$-dimensional,
its semisimplification is the direct sum of two associated characters
$\psi_1$ and $\psi_2$.
If $\psi_1^{12} \notin \{1, \cyc_\ell^6, \cyc_\ell^{12}\}$, then
by Lemma~\ref{mainthmforcurves}, Case~\ref{ellcurve1} holds.
Hence, it remains to show that Case~\ref{ellcurve2} holds
when $\psi_1^{12} \in \{1, \cyc_\ell^6, \cyc_\ell^{12}\}$.

\paragraph{\boldmath Case 1: $\psi_1^{12} \in \{1, \cyc_\ell^{12}\}$.}
If $\psi_1^{12} = \cyc_\ell^{12}$, then the Weil pairing gives
\[\psi_2^{12} = (\cyc_\ell \otimes \psi_1^{-1})^{12} = \cyc_\ell^{12} \otimes (\psi_1^{12})^{-1} = 1.\]
Thus, by interchanging indices if necessary,
we can assume without loss of generality that $\psi_1^{12}$ is trivial.

Then, $\psi_1$ defines a degree (at most) $12$ extension $M$ of $K$.
By construction, the Galois group $\gal(K^\text{ab}/M)$
is killed by $\psi_1$,
so when we consider $E$ as a curve over $M$,
the character $\psi_1$ is trivial.
Thus, we have a Galois-invariant subspace $V\subset E[\ell]$
such that either $V$ is pointwise fixed by
$G_M=\gal(\bar{K}/M)$, or the quotient $E[\ell] / V$
is pointwise fixed by $G_M$.
In the first case, $E$ has an $\ell$-torsion
point defined over $M$, and in the second case,
the isogenous curve $E / V$ has an $\ell$-torsion point defined
over $M$. Writing $n_M \leq 12 n_K$ for the degree of $M$,
Merel's Theorem \cite{merel} implies that
\[\ell \leq \left(\sqrt{3^{n_M}} + 1\right)^2 \leq \left(3^{6n_K} + 1\right)^2.\]
Thus, so long as we choose $S_K$ to contain all primes at most
$\left(3^{6n_K} + 1\right)^2$, we are done.

\paragraph{\boldmath Case 3: $\psi_1^{12} = \cyc_\ell^6$.}
The Weil pairing implies
\[\psi_2^{12} = (\cyc_\ell \otimes \psi_1^{-1})^{12} = \cyc_\ell^{12} \otimes (\psi_1^{12})^{-1} = \cyc_\ell^{12} \otimes (\cyc_\ell^6)^{-1} = \cyc_\ell^6.\]
In other words,
\[\tilde{\rho}_{E, \ell}^{12} = \cyc_\ell^6 \oplus \cyc_\ell^6.\]

If $\rho_{E, \ell}$ is irreducible over $\ff_\ell$, then
$\tilde{\rho}_{E, \ell} = \rho_{E, \ell}$, so the projective
image of $\rho_{E, \ell}$ in $\pgl_2(\ff_\ell)$ has order at
most $6$. But by Lemma~$18'$ of \cite{quelques}
(which is stated for $K = \qq$, but the same proof
works as long as $\ell$ is unramified in $K$),
the projective image has an element of order at least $(\ell - 1)/4$.
Hence, as long as we choose $S_K$ to contain all primes 
less than $25$, it follows that
$\rho_{E, \ell}$ is already reducible over $\ff_\ell$.

In particular, $\cyc_\ell^6$ is the $12$th power
of some character valued in $\ff_\ell^\times$.
Since we are assuming $\ell$ is unramified in $K$,
the cyclotomic character surjects onto
$\ff_\ell^\times$. Thus, every $12$th power in $\ff_\ell^\times$
is a $6$th power, so $\ell \equiv 3$ mod $4$.

Now, suppose GRH holds for $K[\sqrt{-\ell}]$.
Then, by Corollary~\ref{cor:avo},
we could find a prime ideal $v$ of $K$
such that $v$ is split in $K[\sqrt{-\ell}]$,
of degree~$1$, does not lie over $3$,
and satisfies the inequality
\begin{align*}
\nm^{K}_\qq(v) &\leq \ccebavo \cdot (\log \Delta_{K[\sqrt{\pm \ell}]} + n_{K[\sqrt{\pm \ell}]} \log 3)^2 \\
&= 4\ccebavo \cdot (2\log \Delta_K + 2n_K \log 3 + n_K \log \ell)^2.
\end{align*}

We claim that $\psi_1(\pi_v) + \psi_2(\pi_v) = 0$,
for any $\ell$ more than some constant depending on $K$ alone. Indeed we have
\[\psi_1(\pi_v) + \psi_2(\pi_v) \equiv \sqrt{\nm^K_\qq(v)} \cdot (\zeta + \bar{\zeta}) \mod \mathfrak{l}\]
for some $12$th root of unity $\zeta$.
Since $v$ does not lie over $3$, the right-hand side
cannot be a nonzero rational number.
Thus, if $\psi_1(\pi_v) + \psi_2(\pi_v) \neq 0$, we have
a nontrivial divisibility condition on $\ell$:
\[\ell \, \left| \, \nm^{\qq\left[\sqrt{\nm^K_\qq(v)} \cdot (\zeta + \bar{\zeta})\right]}_\qq \left(\sqrt{\nm^K_\qq(v)} \cdot (\zeta + \bar{\zeta}) - (\psi_1(\pi_v) + \psi_2(\pi_v))\right)\right. .\]
But under any complex embedding, we have
\[\left|\sqrt{\nm^K_\qq(v)} \cdot (\zeta + \bar{\zeta}) - (\psi_1(\pi_v) + \psi_2(\pi_v))\right| \leq \left(1 + \sqrt{\nm^K_\qq(v)}\right)^2.\]
Since $\qq\left[\sqrt{\nm^K_\qq(v)} \cdot (\zeta + \bar{\zeta})\right]$
is a quadratic field, we have
\[\ell \leq \left(1 + \sqrt{\nm^K_\qq(v)}\right)^4 \leq \big(1 + 2\sqrt{\ccebavo} \cdot (2\log \Delta_K + 2n_K \log 3 + n_K \log \ell)\big)^4.\]

But this is clearly impossible for all $\ell$ larger
than a constant depending on $K$ alone.
Thus, we have $\psi_1(\pi_v) + \psi_2(\pi_v) = 0$. Hence,
\[(\psi_1(\pi_v) - \psi_2(\pi_v))^2 = (\psi_1(\pi_v) + \psi_2(\pi_v))^2 - 4 \psi_1(\pi_v) \psi_2(\pi_v) = - 4 \nm^K_\qq v, \]
which is a quadratic nonresidue mod $\ell$,
contradicting the reducibility of $\rho_{E, \ell}$.
\end{proof}

\begin{cor}\label{isogeny}
Under GRH,
the degrees of prime degree isogenies of elliptic curves over $K$
are bounded uniformly if and only if $K$ does not contain the Hilbert class field
of an imaginary quadratic field $F$ (i.e.\ if and only if there
are no elliptic curves with CM defined over $K$).
\end{cor}

\begin{proof} 
If $K$ does not contain the Hilbert class field of an imaginary
quadratic field $F$, then there are no CM elliptic curves
which are defined over $K$ and whose CM-field is contained in $K$.
Thus, for $\ell \notin S_K$, we have that $\rho_{E, \ell}$
is absolutely irreducible for any elliptic curve $E$
defined over $K$. In particular, $E$ cannot admit an isogeny
of degree $\ell$.

Conversely, if $K$ contains the Hilbert class field of an imaginary
quadratic field $F$, then there is a CM elliptic curve defined
over $K$, whose CM-field is contained in $K$. But such a curve has
an isogeny of degree $\ell$ for all primes $\ell$ split in $F$.
\end{proof}

\section{Effective Results \label{sec:effective}}

In this section, we prove Theorem~\ref{effectivethm},
which makes Theorems~\ref{mainthm} and \ref{ellcurve}
(as well as Corollary~\ref{isogeny}) effective.
The method of proof does not depend essentially on GRH;
we only use GRH in Subsection~\ref{subsec:grheff}
when putting everything together
to get the final bound.

\subsection{\label{sec:cg} \boldmath The Quantity $c_g$}

\begin{lm} \label{cyclotomicfield}
Suppose $n$ is a positive integer and $[\qq[\zeta_n] : \qq] \leq 2g$.
Then $n \mid c_g$.
\end{lm}
\begin{proof}
It suffices to prove this Lemma in the case where
$n = q^k$ is a prime power. For every prime $p \neq q$,
we have a natural symplectic form on $\oo_{\qq[\zeta_n]} / (p)$
defined by
\[(\alpha, \beta) \mapsto \tr^{\qq[\zeta_n]}_\qq \left(\alpha \bar{\beta} - \beta \bar{\alpha}\right).\]
Multiplication by $\zeta_n$ thus defines an element
of $\spg_{[\qq[\zeta_n] : \qq]}(\ff_p)$ of order $n$
for all odd primes $p \neq q$.
Since we have an injection
$\spg_{[\qq[\zeta_n] : \qq]}(\ff_p) \hookrightarrow \spg_{2g}(\ff_p)$,
it follows that $n \mid c_g$ by the definition of $c_g$.
\end{proof}

\begin{thm} \label{thmcg}
We have an explicit formula
\[c_g = \prod_{\substack{
\text{prime powers $p^n$} \\
(p - 1)p^{n - 1} \leq 2g < (p - 1)p^n}}
p^n.\]
In particular, $c_1 = 12$. In
general, $c_g$ can be bounded by an exponential in $g$,
\[c_g \leq \ccg \cdot (7.4)^g \leq \ccg \cdot 8^g\]
for an effectively computable absolute constant $\ccg$.
\end{thm}
\begin{proof}
From the prime number theorem, we have
\[\prod_{\substack{\text{prime powers $p^n$} \\ (p - 1)p^{n - 1} \leq 2g < (p - 1)p^n}} p^n \ = \ e^{2g(1 + o(1))} \leq \ccg \cdot 7.4^g \leq \ccg \cdot 8^g\]
for some effectively computable absolute constant $\ccg$.
Thus, it suffices to verify the formula for $c_g$.
From Lemma~\ref{cyclotomicfield}, it is clear that $c_g$
is divisible by the above product.
To see the reverse implication, we need to show that every
prime power $p^n$ dividing $c_g$ divides the above product.

\paragraph{\boldmath Case 1: $p$ is odd.} By Dirichlet's Theorem,
we can find an odd prime $q \neq p$ whose class modulo $p^n$
generates the cyclic group $(\zz / p^n \zz)^\times$.
Now, suppose there is an element $X \in \spg_{2g}(\ff_q)$
of order $p^n$.
By applying the Frobenius automorphism of $\bar{\ff}_q$
to $X$, we see that if $X$ has an eigenvalue $\omega$,
then it must also have an eigenvalue $\omega^q$.
Therefore, every primitive $(p^n)$th root of unity is an
eigenvalue of $X$.
It follows that $(p - 1) p^{n - 1} = |(\zz / p^n\zz)^\times| \leq 2g$,
which completes the proof.

\paragraph{\boldmath Case 2: $p = 2$.} By Dirichlet's Theorem,
we can find an odd prime $q \equiv 3$ mod $2^n$.
Now, suppose there is an element $X \in \spg_{2g}(\ff_q)$
of order $2^n$.
By applying the Frobenius automorphism of $\bar{\ff}_q$
to $X$, we see that if $X$ has an eigenvalue $\omega$,
then it must also have an eigenvalue $\omega^q$.
Since $X \in \spg_{2g}(\ff_q)$,
we see that if $X$ has an eigenvalue $\omega$,
it must also have an eigenvalue $\omega^{-1}$.

But it is a standard fact from elementary number theory that
$q \equiv 3$ and $-1$ generate the multiplicative group
$(\zz / 2^n \zz)^\times$. Since $X$ has order $2^n$,
it has some eigenvalue
which is a primitive $(2^n)$th root of unity;
therefore, every primitive $(2^n)$th root of unity is an
eigenvalue of $X$.
It follows that $2^{n - 1} = |(\zz / 2^n\zz)^\times| \leq 2g$, which completes the proof.
\end{proof}

\subsection{Balanced Characters \label{sec:actunits}}

In this subsection, we make Lemma~\ref{thetaisbalanced} effective;
we also note that balanced characters have a useful
condition on their archimedean valuations, which will be
helpful for making other arguments effective.

\begin{lm}\label{balancedabsval}
When $\theta^S$ is balanced with $S(\sigma) + S(\tau) = a'$,
then for any archimedean absolute value $|\ |$
on $K$, we have
\[|\theta^S(x)|= \sqrt{\nm^K_\qq(x)}^{a'}.\]
\end{lm}
\begin{proof}
Clear from Definition~\ref{def:good}.
\end{proof}

There is only a finite set of $\ell$ that can possibly 
be associated to a unbalanced character for some abelian variety. In what follows
we will quantify this possible unbalanced set.
For every unbalanced character $\theta^S$, there exists some unit $u_S$
for which $\theta^S(u_S)$ is not a root of unity.
Now, observe that we have a natural action of $\gal(\bar{\qq}/\qq)$
on the elements $S$ of $\zz[\Gamma_K]$.
If $S$ and $S'$ are related under this action,
then $\theta^S(u)$ is Galois-conjugate to $\theta^{S'}(u)$;
in particular, $\theta^S(u)$ is a root of unity if and only if
$\theta^{S'}(u)$ is a root of unity.
This implies that we can choose the $u_S$ so they depend only
on the orbit of $S$ under the action of $\gal(\bar{\qq}/\qq)$.
If we do this, then
\[\bchar(K, g) := \prod_{\text{$\theta^S$ unbalanced}} \left(1 - \theta^S(u_S)^{c_g/e}\right)\]
is a rational integer.
Moreover, by Definition~\ref{corresp}, any prime $\ell$
for which a corresponding character
$\theta^S$ is unbalanced must divide $\bchar(K, g)$.

\begin{lm} \label{bchar}
There exists an effectively computable absolute constant $\cchar$
such that we can choose the $u_S$ so that the following inequality holds:
\[\bchar(K, g) \leq \cchar \cdot \exp \left(\frac{2 \cdot R_K \cdot (r_K + 2)! \cdot (2g \cdot c_g + 1)^{n_K + 1} \cdot (\log n_K)^3}{|\log \log n_K|^3}\right).\]
\end{lm}

\begin{rem} 
The constant $\bchar(K, g)$ can be effectively directly computed for any
field $K$,
and in many interesting cases (such as quadratic imaginary fields)
is small.
\end{rem}

\begin{proof}
Define the multiplicative Minkowski embedding $\mu$ and functions $f^S$
as in the proof of Lemma~\ref{restons}. Suppose that $\theta^S$ is a unbalanced character.

Write $\Lambda_S \subset \Lambda$ for the kernel of $f^S$.
Since $\theta^S$ is unbalanced, it follows that $\Lambda_S \subsetneq \Lambda$.
Observe that by definition, $f^S$ does not vanish on any element of
the quotient lattice $\Lambda / \Lambda_S$.
Therefore, to choose $u_S$ so that $f^S(u_S)$
is small, it suffices to find a short
lattice vector of the quotient lattice $\Lambda / \Lambda_S$,
which we can do via Minkowski's Theorem if we first bound
\[\vol(\Lambda / \Lambda_S) = \frac{\vol(\Lambda)}{\vol(\Lambda_S)} = \frac{R_K \sqrt{r_K + 1}}{\vol(\Lambda_S)}.\]
(Here, we think about $\Lambda / \Lambda_S$ sitting inside the real
vector space $\rr_0^{r_K + 1} / \langle \Lambda_S \rangle$.
We give $\rr_0^{r_K + 1} / \langle \Lambda_S \rangle$
an inner product by identifying it with
the orthogonal complement of $\Lambda_S$
and taking the induced inner product from $\rr_0^{r_K + 1}$.)
To do this, we use the following theorem of Voutier:

\begin{cthm}[Voutier \cite{voutier}]
For any $u \in \oo_K^\times$ which is not a root of unity, we have the inequality
\[\log\left(\prod_{i = 1}^{n_K} \max(1, |\sigma_i(u)|)\right) \geq \alpha \qquad \text{where} \qquad \alpha = \frac{1}{4} \left(\frac{\log\log n_K}{\log n_K}\right)^3.\]
\end{cthm}

It follows that the length of any unit in $\Lambda$
under the $L^1$ norm (i.e.\ the sum of the absolute values
of the coordinates) satisfies
\[||\mu(u)||_{L^1} = \sum_{i = 1}^{r_1} \big|\log |\sigma_i(u)|\big| + \sum_{i = r_1 + 1}^{r_1 + r_2} 2 \big|\log |\sigma_i(u)|\big| = 2 \cdot \log \left(\prod_{i = 1}^{n_K} \max(1, |\sigma_i(u)|)\right) \geq 2\alpha.\]
(For the second equality above,
we have used $\sum_{i = 1}^{n_K} \log |\sigma_i(u)| = 0$.)
Now, extend the lattice $\Lambda_S$ to a lattice $\Lambda_S'$
of dimension $r_K + 1$ by adding more basis vectors which
are mutually orthogonal, orthogonal to $\Lambda_S$,
and whose length (under the Euclidean metric) equals
$2\alpha$. From the above Theorem, the (open) unit $L^1$-ball of
radius $2\alpha$ intersects the lattice $\Lambda_S'$
only at the origin. Therefore, by Minkowski's Theorem,
\[\vol(\Lambda_S') \geq \frac{1}{2^{r_K + 1}} \cdot \left(\frac{2^{r_K + 1}}{(r_K + 1)!} \cdot (2\alpha)^{r_K + 1}\right) = \frac{(2\alpha)^{r_K + 1}}{(r_K + 1)!}.\]
Now, write $c$ for the codimension of $\Lambda_S$ in $\Lambda$.
Then, we have
\[\vol(\Lambda_S) =  \frac{1}{(2\alpha)^{c + 1}} \cdot \vol(\Lambda_S') \geq \frac{1}{(2\alpha)^{c + 1}} \cdot \frac{(2\alpha)^{r_K + 1}}{(r_K + 1)!} = \frac{(2\alpha)^{r_K - c}}{(r_K + 1)!}\]
which gives
\[\vol(\Lambda / \Lambda_S) = \frac{\vol(\Lambda)}{\vol(\Lambda_S)} \leq \frac{R_K \cdot \sqrt{r_K + 1} \cdot (r_K + 1)!}{(2\alpha)^{r_K - c}}.\]

Applying Minkowski's Theorem again, we can find some vector
of $\Lambda / \Lambda_S$ whose length under the Euclidean
metric 
(we use the Euclidean metric because the $L^1$ metric on $\rr^{r_K + 1}$
does not induce the $L^1$ metric on $\rr^{r_K + 1}_0 / \langle\Lambda_S\rangle$)
is bounded by (here, we use that the volume of the unit
Euclidean $c$-ball is greater than that of the unit $L^1$-ball,
$2^c / c!$)
\begin{align*}
\left(\vol(\Lambda / \Lambda_S) \cdot c!\right)^\frac{1}{c} &\leq \left(\frac{R_K \cdot \sqrt{r_K + 1} \cdot (r_K + 1)!}{(2\alpha)^c} \cdot c!\right)^\frac{1}{c} \\
&= \frac{1}{2\alpha} \cdot \left(R_K \cdot \sqrt{r_K + 1} \cdot (r_K + 1)! \cdot c!\right)^\frac{1}{c} \\
\intertext{The above function is a decreasing function of $c$,
since $R_K \geq 1/5$ by \cite{reg}. Thus,}
&\leq \frac{R_K \cdot \sqrt{r_K + 1} \cdot (r_K + 1)!}{2\alpha}.
\end{align*}
Thus, we can find some vector
of $\Lambda / \Lambda_S$ whose length under the metric induced
from the $L^1$ metric on $\rr^{r_K + 1}$
is bounded by
\[\sqrt{r_K + 1} \cdot \frac{R_K \cdot \sqrt{r_K + 1} \cdot (r_K + 1)!}{2\alpha} \leq \frac{R_K \cdot (r_K + 2)!}{2\alpha}.\]
Now, observe that $|f^S(u)| \leq 2g \cdot c_g \cdot ||\mu(u)||_{L^1}$ for any $S$.
Therefore, for any unbalanced character $\theta^S$,
we can select $u_S$ such that $\theta^S(u_S)$ is not a root of unity, and
\[ \big| \log |\theta^{S'}(u_S)^{c_g/e}| \big| \leq M \qquad \text{where} \qquad M := 2g \cdot c_g \cdot \frac{R_K \cdot (r_K + 2)!}{2\alpha} \]
where $S'$ is any subset of $\emb{K}$.
This gives
\begin{align*}
\bchar(K, g) &= \prod_{\substack{S \subset \emb{K} \\ \text{$\theta^S$ is unbalanced}}} \left(1 - \theta^S(u_S)^{c_g/e}\right) \\
&\leq \prod_{\substack{S \subset \emb{K} \\ \text{$\theta^S$ is unbalanced}}} \left(1 + |\theta^S(u_S)^{c_g/e}| \right) \\
&\leq (1 + \exp(M))^{(2g \cdot c_g + 1)^{n_K}}.
\end{align*}
Since $\exp(M) \geq \cbar \cdot (2g \cdot c_g + 1)^{n_K}$
for some effectively computable absolute constant $\cbar$,
we have for some effectively computable absolute constant $\cchar$,
\begin{align*}
\bchar(K, g) &\leq \cchar \cdot \exp(M \cdot (2g \cdot c_g + 1)^{n_K}) \\
&= \cchar \cdot \exp \left(\frac{2 \cdot R_K \cdot (r_K + 2)! \cdot (2g\cdot c_g + 1)^{n_K + 1} \cdot (\log n_K)^3}{(\log \log n_K)^3}\right). && \qedhere
\end{align*}
\end{proof}

For the remainder of this paper, we suppose that
$\ell \nmid \bchar(K, g)$; in particular, this implies
that $\theta^S$ is a balanced character.

\subsection{Bounds for Theorem~\ref{thm:charclass}}

We begin this section by bounding the number of possible
values of $\psi_\cc(v)$.

\begin{lm} \label{allabelian}
Let $v \subset \oo_K$ be a prime ideal,
with $\nm^K_\qq v \leq V$.
Then, as $A$ ranges over all abelian varieties of dimension $g$,
$d$ ranges over all integers between $0$ and $2g$,
and $\ell$ ranges over all rational primes coprime to $x$,
there are at most
\[\bposs(K, g, V) = \cpossfrob \cdot \left(256 \cdot V\right)^{\frac{g(g + 1)}{4}}\]
possible values of $\psi_\cc(v)$, for some effectively
computable absolute constant $\cpossfrob$. Moreover,
the magnitude of $\psi_\cc(v)$ under any complex
embedding is bounded by $V^g$.
\end{lm}
\begin{proof}
First, we count the number of possible polynomials
$P_\pi$ satisfying Lemma~\ref{frobpoly},
assuming that all of its roots have magnitude
equal to the square root of the norm of $v$, under any complex embedding.
The first $g + 1$ coefficients of $P_\pi$
determine the rest, since the roots of $P_\pi$ come in pairs
which multiply to the norm of $v$.
Thus, to find the total number of possible such polynomials,
we can multiply
together the number of possible values for the coefficient
of $z^{2g - i}$ in $P_\pi$ for $0 \leq i \leq g$.
This gives
\begin{align*}
\prod_{0 \leq i \leq g} \left(2 \cdot \binom{2g}{i} \cdot \left(\sqrt{\nm^K_\qq (x)}\right)^i\right) &= 2^{g + 1} \cdot \left(\nm^K_\qq (x)\right)^{\frac{g(g + 1)}{4}} \cdot \prod_{0 \leq i \leq 3} \binom{2g}{i} \cdot \prod_{4 \leq i \leq g} \binom{2g}{i} \\
&\leq 2^{g + 1} \cdot \left(\nm^K_\qq (x)\right)^{\frac{g(g + 1)}{4}} \cdot \prod_{0 \leq i \leq 3} \binom{2g}{i} \cdot \left(2^{2g}\right)^{g - 3} \\
&\leq \frac{\ctmp}{64^g} \cdot \left(256 \cdot \nm^K_\qq (x)\right)^{\frac{g(g + 1)}{4}}
\end{align*}
for some effectively computable absolute constant $\ctmp$.

Now, any value of $\psi_\cc(v)$
can be written as a product of distinct roots of such a polynomial,
times a $c_g$th root of unity, times a power of $\nm^K_\qq v$
which is at most $g$. Therefore, the number of possible values of
$\psi_\cc(v)$ is at most the product of the number of possible
polynomials calculated above and
$2^{2g} \cdot c_g \cdot (g + 1) \leq \ccg \cdot 64^g$.
This implies the statement of this lemma.
\end{proof}

\noindent
Let us fix a direct sum decomposition
\[\cla(K) \simeq \zz / h_K'\zz \oplus H(K),\]
for a subgroup $H(K) \subset \cla(K)$, and fix a generator
$\alpha$ of $\zz / h_K' \zz$.
From elementary group theory, the image of any homomorphism
from $\cla(K)$ to a cyclic group is generated by the image
of some element in $\alpha + H(K)$.

Choose prime ideals
$\{v_1, v_2, \ldots, v_{h_K / h_K'}\}$
which are of degree~$1$, coprime to $c_g \cdot h_K$,
and represent each element of $\alpha + H(K)$.
Also, write $(v_i)^{h_K'} = (x_i)$.

In order to make Theorem~\ref{thm:charclass}
effective, we note that if $\psi$ is an associated character
which does not satisfy the conclusion of Theorem~\ref{thm:charclass},
then our proof shows that either $\theta^S$ is unbalanced,
or there is some $i$ so that
Lemma~\ref{equality} fails for $v_i$.
Thus the quantity $\brat(K, g)$, which we define to be
the product of
\[\prod_{\text{balanced characters $\theta^S$}}\left(\prod_{i = 1}^{h_K / h_K'} \prod_{\substack{\text{possible values of $\psi(\pi_{v_i})$} \\ \psi(\pi_{v_i})^{c_g \cdot h_K'} \neq \theta^S(x_i)^{c_g/e}}} \left(\psi(\pi_{v_i})^{c_g \cdot h_K'} - \theta^S(x_i)^{c_g/e}\right)\right)\]
with all primes which lie under one of the $v_i$
or divide $\Delta_K$,
is a rational integer with the property
that we can take $S_{K, g}$ to consist of all primes
dividing $\brat(K, g) \cdot \bchar(K, g)$.

\begin{lm} \label{lm:brat} We have
\[\brat(K, g) \leq \left(2 \cdot V^{2g \cdot c_g \cdot h_K'}\right)^{\sqrt{\left(2g \cdot c_g + 1\right)^{n_K}} \cdot (h_K/h_K') \cdot \bposs(K, g, V)} \cdot V^{h_K / h_K'} \cdot \Delta_K,\]
where $V$ is the maximum norm of the $v_i$.
\end{lm}

\begin{proof} Under any complex embedding, Lemma~\ref{balancedabsval}
implies that
\begin{align*}
\left|\psi(\pi_{v_i})^{c_g \cdot h_K'} - \theta^S(x_i)^{c_g/e}\right| &\leq \left|\psi(\pi_{v_i})^{c_g \cdot h_K'}\right| + \left|\theta^S(x_i)^{c_g/e}\right| \\
&\leq 2 \cdot |\nm^K_\qq v_i|^{2g \cdot c_g \cdot h_K'} \\
&\leq 2 \cdot V^{2g \cdot c_g \cdot h_K'}
\end{align*}
which completes the proof, using Lemma~\ref{allabelian},
together with the fact that there at most
$\sqrt{(2g \cdot c_g + 1)^{n_K}}$
balanced characters. (The factor of $V^{h_K / h_K'}$ is there to
account for the primes which lie under one of the $v_i$.)
\end{proof}

\subsection{\label{subsec:grheff} Proof of Theorem~\ref{effectivethm}}

In this subsection, we prove Theorem~\ref{effectivethm}.
While Theorem~\ref{mainthm} is true unconditionally, assuming
GRH allows us to get a significantly better bound.

\begin{rem}
Using an unconditional version of the Chebotarev Density Theorem
(see Remark~\ref{remuncond}),
one could make Theorem~\ref{effectivethm} unconditional using
the same method.
In the case of $g = 1$, there is also an unconditional
bound due to David; see Theorem~2 of \cite{david}.
(The logarithm of David's bound is roughly the $n_K$th power
of the logarithm of the conditional bound given here.)
\end{rem}

\begin{thm}\label{effectivethm}
Under GRH,
there are effectively computable absolute constants
$\cell$, $\cprod$, and $\cblah$ such that
we can take in
Theorems~\ref{ellcurve} and~\ref{mainthm}
\begin{align*}
\prod_{\ell \in S_K} \ell &\leq \exp\Big(\cell^{n_K} \cdot \left(R_K \cdot n_K^{r_K} + h_K^2 \cdot (\log \Delta_K)^2 \right)\Big) \\
\prod_{\ell \in S_{K, g}} \ell &\leq \exp\left(\cprod^{n_K} \cdot \Big(
8^{g(n_K + 1)} \cdot R_K \cdot n_K^{r_K} + 3^{g \cdot n_K} \cdot \big(\cblah \cdot g \cdot h_K \cdot n_K \cdot \log \Delta_K\big)^{\frac{g(g + 1)}{2} + 1}\Big)\right).
\end{align*}
\end{thm}

\begin{rem} \label{besteff} In fact, we have proven (independently of GRH) that
\[\prod_{\ell \in S_{K, g}} \ell \leq \bchar(K, g) \cdot \brat(K, g).\]
\end{rem}

\begin{proof}
Under GRH, Corollary~\ref{cor:avo}
(applied to $N = c_g \cdot h_K$),
and $\log h_K \leq \frac{3}{2} \log \Delta_K$
(which follows from \cite{lenstra}, Theorem~6.5),
implies that we can choose the $v_i$ so that
\[V = \max \nm^K_\qq(v_i) \leq \ccebavo \cdot \big(\log \Delta_{H_K} + n_{H_K} \log (c_g \cdot h_K)\big)^2 \leq \cvbound \cdot g^2 \cdot h_K^2 \cdot n_K^2 \cdot (\log \Delta_K)^2\]
for some effectively computable absolute constant $\cvbound$.

We use the notation $f \lesssim g$ to mean that
$f \leq C \cdot g$ for an effectively computable absolute
constant $C$.

By Remark~\ref{besteff} above, it suffices to show
both of $\log \bchar(K, g)$ and $\log \brat(K, g)$
are bounded by a constant times
the logarithm of the right-hand side of the first inequality
appearing in the statement of the theorem.
First, we bound $\bchar(K, g)$, as follows.
\[\log \bchar(K, g) \lesssim \frac{R_K \cdot (r_K + 2)! \cdot (2g \cdot c_g + 1)^{n_K + 1} \cdot (\log n_K)^3}{|\log \log n_K|^3} \lesssim 8^{g(n_K + 1)} \cdot R_K \cdot (\cprod n_K)^{r_K}\]
for some effectively computable absolute constant $\cprod$.
Next, we bound $\brat(K, g)$, using $\log h_K \leq \frac{3}{2} \log \Delta_K$:
\begin{align*}
\log \brat(K, g) &\lesssim g \cdot c_g \cdot h_K' \cdot \sqrt{(2g \cdot c_g + 1)^{n_K}} \cdot (\log g + \log h_K + \log n_K + \log \log \Delta_K) \\
&\qquad \cdot (h_K / h_K') \cdot \left(256 \cdot \cvbound \cdot g^2 \cdot h_K^2 \cdot n_K^2 \cdot (\log \Delta_K)^2\right)^{\frac{g(g+1)}{4}} \\
&\lesssim (\cblahblah \cdot 3^{g})^{n_K} \cdot \big(\cblah \cdot g \cdot h_K \cdot n_K \cdot \log \Delta_K\big)^{\frac{g(g+1)}{2} + 1}
\end{align*}
for effectively computable absolute constants $\cblah$ and $\cblahblah$.

Now, from the proof of Theorem~\ref{ellcurve}, it suffices
to show that the logarithm of the product of all primes $\ell$ for which
\[\ell \leq \left(3^{12n_K} + 1\right)^2 \quad \text{or} \quad \ell \leq \big(1 + 2\sqrt{\ccebavo} \cdot (2\log \Delta_K + 2n_K \log 3 + n_K \log \ell)\big)^8\]
is bounded by an absolute constant times the logarithm of the
right-hand side of the second inequality appearing in the statement
of the theorem. For the product of all primes satisfying
the first of the above two inequalities, this is clear from
the prime number theorem.

For the second factor, note that the
Brauer-Siegel theorem (which is effective under GRH) implies that
\[\cell^{n_K} \cdot \left(R_K \cdot n_K^{r_K} + h_K^2 \cdot (\log \Delta_K)^4 \right) \gtrsim h_K^2 + n_K^{n_K/2} \cdot R_K \gtrsim n_K^{n_K/3} \cdot (h_K R_K)^{2/3} \gtrsim \Delta_K^{1/4}.\]
Thus, it suffices to show that for $\ell \geq \sqrt[4]{\Delta_K}$ and more than
an effectively computable absolute constant, we have
\[\ell \geq \big(1 + 2\sqrt{\ccebavo} \cdot (2\log \Delta_K + 2n_K \log 3 + n_K \log \ell)\big)^4.\]
But this is clear, since Minkowski's bound together with the assumption
that $\ell \geq \sqrt[4]{\Delta_K}$ imply that the right-hand side of the above
inequality is at most an effectively computable absolute constant
times $(\log \ell)^8$.
This completes the proof.
\end{proof}

\bibliography{varieties-bibliography}{}
\bibliographystyle{plain}

\appendix

\section{\label{detcrys} A Determinantal Comparison (by Brian Conrad)}

\input{detcrys.tex}

\end{document}

%% file: detcrys.tex
\subsection{Motivation}

For a linear representation $\rho$ of
a group $\Gamma$ on a finitely generated module $V = \prod V_i$ 
over a finite product $\prod F_i$ of fields $F_i$, 
we get a determinant $\det \rho\colon \Gamma \to \prod F_i^{\times}$
via the $\Gamma$-action on $\prod \det(V_i)$.
(This is the usual $(\prod F_i)$-linear determinant when $V$ is free as a $(\prod F_i)$-module.)
The case of interest to us
will be $\Gamma = G_K := \gal(K_s/K)$
for a finite extension $K$ of
$\qq_p$  and the action of $G_K$ on $V = \vv_p(A)$
for an abelian variety $A$ of dimension $g > 0$ over $K$.

Consider a commutative subfield $F \subseteq \mend_K^0(A)$
(which could even be $\qq$, though
that case will not be interesting), so for $F_p := \qq_p \otimes_{\qq} F$,  
there is a natural continuous $F_p$-linear representation $\rho_{A,p}$
of $G_K$ on $\vv_p(A)$.  This 
yields a continuous determinant homomorphism
$G_K^{\rm{ab}} \to F_p^{\times}$.
Composing with
the Artin map $r_K\colon K^{\times} \to G_K^{\rm{ab}}$,
we get a composite map
\[\begin{CD}
\psi_A \colon K^\times @>{r_K}>> G_K^{\rm{ab}} @>{\det_{F_p} \rho_{A,p}}>> F_p^{\times}.
\end{CD}\]

A natural question, inspired by the use of the reflex norm in the Main Theorem
of complex multiplication, 
is to ask whether the restriction of $\psi_A$
to an open subgroup of $\oo_K^{\times}$ 
can be described in terms of 
the $F$-action on $\lie(A)$. 
To be precise, $\lie(A)$ is naturally a module over $K \otimes_{\qq} F = K \otimes_{\qq_p} F_p$,
so $K$ acts $F_p$-linearly on $\lie(A)$ and hence we can take the $F_p$-linear
determinant of (the inverse of) the 
$K^{\times}$-action on $\lie(A)$:
\[\chi_A(a) = \det_{F_p} \big((x \mapsto a \cdot x) \colon \lie(A) \to \lie(A) \big).\]

It is natural to ask if the homomorphisms $\psi_{A}, \chi_A\colon
K^{\times} \rightrightarrows F_p^{\times}$
are equal on an open subgroup of $\oo_K^{\times}$
when $F$ is a totally real or CM field.  The answer is affirmative
without archimedean restrictions on $F$, and 
by using $p$-adic Hodge theory in the form due to Fontaine
(see \cite{FontainePeriods}, \cite{FontaineSST}), which goes 
beyond what was used by Serre and Tate (who worked
in the area before the discovery of $B_{\rm{crys}}$) we can say a bit more:

\begin{thm}\label{maps} The two maps $K^{\times} \rightrightarrows F_p^{\times}$ coincide
on an open subgroup of $\oo_K^{\times}$.  If $A$ has semistable reduction then these maps 
agree on $\oo_K^{\times}$.
\end{thm}

\begin{exe} If $A$ is a CM abelian variety with good reduction and $F$ is a CM field with
$[F:\qq] = 2g$ then this theorem recovers the inertial description of 
the reflex norm in the theory of complex multiplication when there is good reduction.
\end{exe}

We will use Grothendieck's orthogonality theorem
in the semistable case to reduce Theorem \ref{maps} to a more general assertion about 
$p$-divisible groups over the valuation ring of a finite extension of $K$.
The main point is to then recast this general assertion in terms of $p$-adic
Hodge theory, since that admits a robust tensorial structure
whereas $p$-divisible groups do not.  In what follows,
we will use the covariant Fontaine functors
(i.e., $(B \otimes_{\qq_p} V)^{G_K}$ rather than $\hom_{\qq_p[G_K]}(V,B)$).

\subsection{\boldmath Reformulation via $p$-divisible groups}

To prove Theorem \ref{maps}, it is harmless to 
replace $K$ with a finite extension so that $A$
has semistable reduction.  So we now assume this to be the case. 
Let $\cala$ denote the semi-abelian relative identity component
of the N\'eron model $N(A)$ over $\oo_K$ (i.e., the open complement in
$N(A)$ of the union of the non-identity components of the special fiber $N(A)_k$). 
The special fiber $\cala_k$ is an extension of an abelian variety $B$ by a torus $T$.  We will
use a filtration on $\vv_p(A)$ arising from the filtration on $\cala_k$ to reduce
our problem to an intrinsic question about $p$-divisible groups over $\oo_K$.

Now we recall Grothendieck's results on the structure
of $p$-adic Tate modules of abelian varieties with semistable reduction
over $p$-adic fields.  (This is developed in [SGA 7, Exp.\ IX];
see \cite[\S4--\S5]{mordell} for an exposition, especially \cite[Thm.\,5.5]{mordell}.)  Let
$a = \dim B$ and $t = \dim T$.  
The ``finite parts'' of the $\cala[p^n]$ (according to the decomposition of the quasi-finite flat separated
$\oo_K$-groups 
$\cala[p^n]$ via Zariski's Main Theorem) define a $p$-divisible group $\Gamma$ over
$\oo_K$ of height $2a + t$ lifting
$\cala_k[p^{\infty}]$.  This has generic fiber contained in
$A[p^{\infty}]$, and it is final among $p$-divisible groups over $\oo_K$ whose
generic fiber is equipped with a map to $A[p^{\infty}]$.
By Grothendieck's orthogonality theorem,  the quotient $\vv_p(A)/\vv_p(\Gamma)$
is the Galois representation
associated to the Cartier dual of the unique $p$-divisible group $\Gamma'_{\rm{t}}$ over $\oo_K$ 
lifting the $p$-divisible group $T'[p^{\infty}]$ of the maximal torus $T'$ in the special fiber of
the N\'eron model of the dual abelian variety $A'$.  Since
$\Gamma'_{\rm{t}}$ has \'etale Cartier dual (as this
holds for its special fiber $T'[p^{\infty}]$), it follows
that $\vv_p(A)/\vv_p(\Gamma)$ is {\em unramified}.  Hence, 
the inertial restriction of
$\psi_A$ is unaffected by replacing 
$\vv_p(A)$ with $\vv_p(\Gamma)$.  

The Lie algebra of
$A$ is the generic fiber of  $\lie(\cala)$, and
$\lie(\cala)$ is naturally identified with the Lie algebra of the formal $\oo_K$ group 
$\widehat{\cala} = {\rm{Spf}}(\oo^{\wedge}_{\cala,0})$ of $\cala$
(completion along the identity section over $\oo_K$).  But $\widehat{\cala}$
is the formal group over $\oo_K$ corresponding
to the connected part of the $p$-divisible group $\Gamma$
under the Serre--Tate equivalence between connected
$p$-divisible groups and commutative formal Lie groups
on which $[p]$ is an isogeny (over any complete
local noetherian ring with residue characteristic $p$).  
Hence, $\lie(A) = \lie(\Gamma)[1/p]$ functorially
in the isogeny category over $K$.  (Note
that $\lie(\Gamma)[1/p]$ is functorial
with respect to $K$-homomorphisms in
$\Gamma$, due to Tate's isogeny theorem for
$p$-divisible groups over $\oo_K$.) 
The $F_p$-action on $\vv_p(\Gamma)$
arises from an $F_p$-action on $\Gamma$ in the isogeny
category over $\oo_K$. 

We conclude that our entire problem is intrinsic to the $p$-divisible group $\Gamma$
over $\oo_K$, in the sense that it involves relating
the inertial action on $\det_{F_p}(\vv_p(\Gamma))$
to the $F_p$-determinant of the
$K^{\times}$-action on $\lie(\Gamma)[1/p]$. 
In this way, our problem makes sense more generally for
an arbitrary $p$-divisible group over $\oo_K$
equipped with an action by $F_p$ in the isogeny category over $\oo_K$. 
Decomposing $\Gamma$ (up to isogeny over $\oo_K$) according to the idempotents of $F_p$,
and renaming each factor field of $F_p$ as $F$, thereby
reduces Theorem \ref{maps} to: 

\begin{thm}\label{pdiv} Let $K$ be a finite extension of
$\qq_p$, $\Gamma$ a $p$-divisible group over $\oo_K$, 
and $F$ a finite
extension of $\qq_p$ equipped with an action on $\Gamma$ in the isogeny category
over $\oo_K$.  

Let $\chi\colon K^{\times} \to F^{\times}$
be defined by the reciprocal of the
$F$-linear determinant of the $K^{\times}$-action on $\lie(\Gamma)[1/p]$,
and let the composite map 
\[\psi\colon \begin{CD}
K^{\times} @>{r_K}>>  G_K^{\rm{ab}} @>>> F^{\times}
\end{CD}\]
be defined by the $F$-linear determinant of the $G_K$-action on $\vv_p(\Gamma)$.
Then $\psi|_{\oo_K^{\times}} = \chi|_{\oo_K^{\times}}$.
\end{thm}

\subsection{Proof of Theorem \ref{pdiv}}

In view of how $\chi$ is constructed from a $(K \otimes_{\qq_p} F)$-module, 
it arises from a homomorphism of
$\qq_p$-tori
$\res^K_{\qq_p} \gm{K} \to \res^F_{\qq_p} \gm{F}$. Thus, by 
\cite[Prop.\,B.4(i)]{conrad}, $\chi|_{\oo_K^{\times}}$ 
is the $I_K$-restriction of a crystalline representation $G_K^{\rm{ab}} \to F^{\times}$. 
Hence, if $\psi$ and $\chi$ agree on an open subgroup of $\oo_K^{\times}$
then their ratio on $\oo_K^{\times}$ is
the $I_K$-restriction of a crystalline representation that is finite on $I_K$.  But
a crystalline $p$-adic representation of $G_K$ with finite
image on $I_K$ is unramified, so it would follow that $\chi$ and $\psi$ coincide
on $\oo_K^{\times}$.  In particular, if 
$\chi^e$ and $\psi^e$ coincide on $\oo_K^{\times}$ for some $e > 0$
(so $\chi$ and $\psi$ agree on the open subgroup
$(\oo_K^{\times})^e$) then we will be done.  

It is harmless to replace $\Gamma$ with an $\oo_K$-isogenous $p$-divisible group, so we may and do
assume that $\oo_F$ acts on $\Gamma$ (not just in the isogeny category).  
If $F'/F$ is a finite extension then it is harmless to replace $\Gamma$ with its power
$\oo_{F'} \otimes_{\oo_F} \Gamma$ (defined in the evident manner), since
at the determinant level we would be replacing $\chi$ and $\psi$
with their $[F':F]$th powers, which we have seen is harmless.  Thus, we may and do arrange that 
$F$ splits $K/\qq_p$.

Let $\Gamma^{\vee}$ denote the dual of $\Gamma$, and consider
 the $\cc_K$-linear $G_K$-equivariant canonical Hodge--Tate decomposition
$\cc_K \otimes_{\qq_p} \vv_p(\Gamma) \simeq (\cc_K(1) \otimes_K t_{\Gamma}) \oplus
(\cc_K \otimes_K \hom_K(t_{\Gamma^{\vee}}, K))$, 
where $t_{\Gamma} := \lie(\Gamma)[1/p]$ (and similarly for $t_{\Gamma^{\vee}}$).
For later purposes, it will be convenient to apply the follow elementary lemma to rewrite
the second summand.

\begin{lm} Let $K$ and $F$ be finite separable extensions
of a field $k$.  For any finitely generated $(K \otimes_k F)$-module $W$,
the $(K \otimes_k F)$-modules $\hom_K(W,K)$ and $\hom_F(W,F)$
are naturally isomorphic, where
$F$ acts $K$-linearly on $\hom_K(W,K)$ through
functoriality applied to its $K$-linear action on $W$
and similarly for the $F$-linear $K$-action on $\hom_F(W,F)$.
\end{lm}

\begin{proof}
It suffices to prove that the natural $(K \otimes_k F)$-linear map
\[\hom_K(W,K) \to \hom_k(W,k) \quad \text{defined via} \quad \ell \mapsto \tr^K_k \circ \ell\]
is an isomorphism,
as then we can argue similarly with the roles of $K$ and $F$ swapped. 
This only involves the underlying $K$-vector space of $W$
(ignoring the $F$-action), so we can reduce to the trivial case
$W = K$.
\end{proof}

\noindent
We now rewrite the Hodge--Tate decomposition in the form
\begin{equation}\label{new}
\cc_K \otimes_{\qq_p} \vv_p(\Gamma) \simeq (\cc_K(1) \otimes_K t_{\Gamma}) \oplus
(\cc_K \otimes_K \hom_F(t_{\Gamma^{\vee}}, F)),
\end{equation}
where $\hom_F(t_{\Gamma^{\vee}},F)$ is a
$K$-vector space through functorality applied to the $F$-linear $K$-action on $t_{\Gamma^{\vee}}$. 
Since $F$ splits $K/\qq_p$, any $(K \otimes_{\qq_p} F)$-module $W$ 
(such as $t_{\Gamma}$ and $t_{\Gamma^{\vee}}$) 
decomposes into $F$-subspaces
\[W = \oplus_{\sigma} W_{\sigma}\]
according to a $\qq_p$-embedding
$\sigma\colon K \to F$ through which $K$ acts.  
That is, for $w \in W_{\sigma}$ we have $(c \otimes 1)w = \sigma(c)w$
for $c \in K$. 
We can therefore compute 
the $(\cc_K \otimes_{\qq_p} F)$-linear determinant on both sides of \eqref{new} 
by first collapsing the $K$-action into the $F$-structure
by decomposing modules into isotypic subspaces according to $\qq_p$-embeddings
$\sigma\colon K \to F$, then decomposing those subspaces into isotypic $\cc_K$-subspaces 
according to the $\qq_p$-embedding $F \to \cc_K$
through which $F$ acts, and then finally forming the $\cc_K$-determinant
of each such subspace of the latter sort.  Thus, 
 the $(\cc_K \otimes_{\qq_p} F)$-determinant of the left side 
of \eqref{new} is $\cc_K \otimes_{\qq_p} \psi = \cc_K \otimes_K (K \otimes_{\qq_p} \psi)$ and the 
$(\cc_K \otimes_{\qq_p} F)$-determinant of the right side of \eqref{new} is 
\begin{equation}\label{rhs}
\bigoplus_{\sigma\colon K \to F} \cc_K \otimes_{K,\sigma} \left(\det_F(t_{\Gamma,\sigma}(1)) \otimes_F
\det_F(\hom_F(t_{\Gamma^{\vee}},F)_{\sigma})\right)
\end{equation}
as $\sigma$ varies through the $\qq_p$-embeddings of $K$ into $F$. 

Let $\theta_{\chi}\colon G_K^{\rm{ab}} \to \oo_F^{\times}$
correspond to a map extending $\chi|_{\oo_K^{\times}}$ via $r_K$, so it suffices
 to prove that $\psi$ and $\theta_{\chi}$ coincide
on an open subgroup of $\oo_K^{\times}$.  It is equivalent to say
that the ratio of these $\oo_F^{\times}$-valued Hodge--Tate characters has finite image
on inertia, or in other words that 
their $\cc_K$-scalar extensions (over $\qq_p$) are $(\cc_K \otimes_{\qq_p} F)$-linearly 
and $G_K$-equivariantly isomorphic. 
In other words, it suffices to prove that $\cc_K \otimes_{\qq_p} \theta_{\chi}$
is $(\cc_K \otimes_{\qq_p} F)$-linearly and $G_K$-equivariantly isomorphic to \eqref{rhs}.
It is harmless to replace $G_K$-equivariance with $H$-equivariance for
an open subgroup $H$, such as the Galois group of $\overline{K}$
over the Galois closure $F'$ of $F/\qq_p$.

Our remaining task is to compute the Hodge--Tate weights of
the $\cc_K$-semilinear $G_{F'}$-representation $\cc_K \otimes_{j,F} \theta_{\chi}$
for each $\qq_p$-embedding $j\colon F \to \cc_K$.  
For any such $j$ the image $j(F)$ contains $K$ and so induces
a $\qq_p$-embedding of $K$ into $F$.  
Since we use covariant Fontaine functors,
the Hodge--Tate weight of $\qq_p(n)$ is $-n$
($B_{\rm{HT}}(n)$ has its $G_K$-invariants occurring in degree $-n$). 
It therefore suffices to prove that for each $\qq_p$-embedding 
\mbox{$\sigma\colon K \to F$}, the $K$-dimension of
the $\sigma$-isotypic part of the 
$(K \otimes_{\qq_p} F)$-module ${\rm{gr}}^n(D_{\rm{dR}}(\theta_{\chi}))$ vanishes
for $n \ne n_{\sigma} := -\dim_F t_{\Gamma,\sigma}$
and is $[F:K]$ for $n = n_{\sigma}$.  This means precisely 
that the $\sigma$-isotypic part of
$D_{\rm{dR}}(\theta_{\chi})$ is 1-dimensional over $F$ with its unique
nonzero ${\rm{gr}}^n$ occurring for $n = n_{\sigma}$.  

Combining these assertions over all $\sigma$, our task is to prove
that $D_{\rm{dR}}(\theta_{\chi})$ free
 of rank 1 over $K \otimes_{\qq_p} F$ and the 
$\sigma$-isotypic part of 
${\rm{gr}}^{\bullet}(D_{\rm{dR}}(\theta_{\chi}))$
is supported in degree $-\dim_F t_{\Gamma,\sigma}$.
By \cite[Prop.\,A.3]{conrad},
$\theta_{\chi}$ is crystalline
and $D_{\rm{crys}}(\theta_{\chi})$ is invertible as
a $(K_0 \otimes_{\qq_p} F)$-module.  Extending
scalars by $K_0 \to K$ yields
$D_{\rm{dR}}(\theta_{\chi})$, so the invertibility over
$K \otimes_{\qq_p} F$ holds.

It remains
to prove that the degree of the $\sigma$-isotypic $F$-line in 
${\rm{gr}}^{\bullet}(D_{\rm{dR}}(\theta_{\chi}))$ is equal to $-\dim_F t_{\Gamma,\sigma}$
for each $\sigma\colon K \to F$. 
Recall that by definition, $\chi\colon K^{\times} \to F^{\times}$ encodes
the $K$-action on the $F$-line 
\[\det_F(t_{\Gamma})^{-1} = \bigotimes_{\sigma\colon K \to F} \det_F(t_{\Gamma,\sigma})^{-1}.\]
In other words, for any $c \in K^{\times}$, $\chi(c) = \prod_{\sigma} \sigma(c)^{n_{\sigma}}$
as a product of $F^{\times}$-valued characters,
i.e.\ $\theta_{\chi} = \otimes_{\sigma} \theta_{\sigma}^{\otimes n_{\sigma}}$
where (i) $\theta_{\sigma}\colon G_K^{\rm{ab}} \to \oo_F^{\times}$
extends $r_K(u) \mapsto \sigma(u)$ for
$u \in \oo_K^{\times}$, and (ii) 
the tensor product is formed as 1-dimensional $F$-linear representations. 
Since we are using covariant Fontaine 
functors, $D_{\rm{dR}}(\theta_{\chi}) \simeq \otimes_{\sigma} D_{\rm{dR}}(\theta_{\sigma})^{\otimes n_{\sigma}}$
where the tensor product is formed over $K \otimes_{\qq_p} F$ and the definition of the
$F$-linear $K$-action on the $\sigma$-factor via
$\sigma\colon K \to F$. 

By Serre \cite[App.\,III, A.4]{serreapp}, the representation
$\theta_{\sigma}^{-1}$ corresponds to the scalar extension along $\sigma$
of a Lubin--Tate group $G_{\pi}$ over $\oo_K$ arising
from a uniformizer $\pi$ of $K$.   
The associated filtered $K$-vector space
$D_{\rm{dR}}(\vv_p(G_{\pi}))$ has 
${\rm{gr}}^{-1}$ of dimension 1
and ${\rm{gr}}^0$ of dimension
$[K:\qq_p]-1$ (since $G_{\pi}$ is 1-dimensional $p$-divisible group 
of height $[K:\qq_p]$ over $\oo_K$, and $\cc_K \otimes D_{\rm{dR}} = D_{\rm{HT}}$).

Using the $G_K$-equivariant $K$-linear structure on $\vv_p(G_{\pi})$,
view $D_{\rm{dR}}(\vv_p(G_{\pi}))$
as a filtered $(K \otimes_{\qq_p} K)$-module with
the left tensor factor encoding the $K$-linear
structure on $D_{\rm{dR}}$ (arising from $B_{\rm{dR}}$)
and the right tensor factor encoding the $K$-action arising
from $\vv_p(G_{\pi})$.  

\begin{lm} As a $(K \otimes_{\qq_p} K)$-module, 
$D_{\rm{dR}}(\vv_p(G_{\pi}))$ is free of rank $1$.
\end{lm}

\begin{proof}
The comparison isomorphism 
$B_{\rm{dR}} \otimes_K D_{\rm{dR}}(\vv_p(G_{\pi})) \simeq
B_{\rm{dR}} \otimes_{\qq_p} \vv_p(G_{\pi})$
has target that is visibly faithful over $K \otimes_{\qq_p} K$.
Hence, $D_{\rm{dR}}(\vv_p(G_{\pi}))$ is a faithful
$(K \otimes_{\qq_p} K)$-module, so 
by $K$-dimension reasons (for the left tensor structure)
it is free of rank 1.
\end{proof}

Using the $K$-structure for the right tensor factor,
$D_{\rm{dR}}(\theta_{\sigma}^{-1}) = D_{\rm{dR}}(\vv_p(G_{\pi})) \otimes_{K,\sigma} F$.
Thus, by the lemma, $D_{\rm{dR}}(\theta_{\sigma})$
is an invertible $(K \otimes_{\qq_p} F)$-module
equipped with a linear filtration whose associated graded
module is supported in degrees $1$ and $0$
with the term in degree $1$ equal to the $\sigma$-isotypic $F$-line
and the term in degree 0 equal to the span of the isotypic $F$-lines for
the other $\qq_p$-embeddings of $K$ into $F$. 
In particular, the $(K \otimes_{\qq_p} F)$-linear structure
canonically splits the filtration (via the decomposition into isotypic $F$-lines
for the $K$-action), so we may and do view $D_{\rm{dR}}(\theta_{\sigma})$
as a graded $(K \otimes_{\qq_p} F)$-module (equipped with the associated
tautologous filtration). 
Hence, $D_{\rm{dR}}(\theta_{\sigma})^{\otimes n_{\sigma}}$ is
an invertible $(K \otimes_{\qq_p} F)$-line equipped
with a linear grading such that the 
$\sigma$-isotypic $F$-line is in degree $n_{\sigma}$
and whose other isotypic $F$-lines are in degree 0
(since the factor rings of $K \otimes_{\qq_p} F$ are pairwise
orthogonal).   

Finally, 
the $(K \otimes_{\qq_p} F)$-linear tensor product over all $\sigma$
gives that $D_{\rm{dR}}(\theta_{\chi})$ is an invertible
$(K \otimes_{\qq_p} F)$-module equipped
with a linear grading such that its
$\sigma$-isotypic $F$-line is the tensor product
of the $\sigma$-isotypic $F$-line
in $D_{\rm{dR}}(\theta_{\sigma})^{\otimes n_{\sigma}}$
(which occurs in degree $n_{\sigma}$)
and the $\sigma$-isotypic $F$-lines
in the $D_{\rm{dR}}(\theta_{\tau})^{\otimes n_{\tau}}$'s
for $\tau \ne \sigma$ (which all occur in degree 0). 
To summarize, for every $\sigma$
the $\sigma$-isotypic $F$-line in
${\rm{gr}}^{\bullet}(D_{\rm{dR}}(\theta_{\chi}))$ occurs in
degree $n_{\sigma} = -\dim_F t_{\Gamma,\sigma}$, as desired.